\documentclass[oneside]{amsart}
\usepackage{amsfonts,amsthm,amsmath}
\usepackage{amssymb}
\usepackage{amscd}
\usepackage[utf8]{inputenc}
\usepackage[a4paper]{geometry}
\usepackage{enumerate}
\usepackage[usenames,dvipsnames]{color}
\usepackage{soul}
\usepackage{array}
\usepackage{array,tabularx}
\usepackage{xypic} 
\usepackage{epic}
\usepackage{float}
\usepackage{graphicx}
\usepackage{caption}
\usepackage{subcaption}

\usepackage[normalem]{ulem}

\usepackage{tikz}
\usetikzlibrary{arrows, intersections, calc, matrix}

\numberwithin{equation}{section}
\numberwithin{equation}{subsection}

\theoremstyle{plain}
\newtheorem{theorem}[equation]{Theorem}
\newtheorem{lemma}[equation]{Lemma}

\newtheorem{corollary}[equation]{Corollary}

\newtheorem{thm}[equation]{Theorem}
\newtheorem{cor}[equation]{Corollary}

\newtheorem{prop}[equation]{Proposition}

\theoremstyle{definition}
\newtheorem{example}[equation]{Example}
\newtheorem{remark}[equation]{Remark}

\newtheorem{definition}[equation]{Definition}

\newcommand{\tX}{\widetilde{X}}

\def\qp#1{\mathfrak{#1}}
\def\I{{I}}

\def\Q{\mathbb Q}
\def\R{\mathbb R}
\def\Z{\mathbb Z}

\newcommand{\cale}{{\mathcal E}}

\newcommand{\calv}{{\mathcal V}}

\newcommand{\cali}{{\mathcal I}}

\newcommand{\calF}{{\mathcal F}}
\newcommand{\calO}{{\mathcal O}}

\newcommand{\calS}{{\mathcal S}}\newcommand{\cS}{{\mathcal S}}

\newcommand{\caln}{\mathcal{N}}

\newcommand{\V}{\calv}

\newcommand{\cl}{\check{\ell}}

\newcommand{\bt}{{\mathbf t}}

\newcommand{\cell}{\check{\ell}}
\newcommand{\igam}{i_{\tiny (\Gamma',\Gamma)}}
\newcommand{\cpi}{\check{\pi}}

\newcommand{\bZ}{{\mathbb{Z}}}
\newcommand{\bQ}{{\mathbb{Q}}}

\newcommand{\zbz}{Z_{B_0}}
\newcommand{\zbe}{Z_{B_1}}

\newcommand{\zm}{Z_{B_m}}
\newcommand{\zi}{Z_{B_i}}
\newcommand{\zj}{Z_{B_j}}

\newcommand{\czk}{[Z_K]}
\newcommand{\szk}{s_{[Z_K]}}

\author{Tam\'as L\'aszl\'o}
\thanks{}
\address{Babe\c{s}-Bolyai University, Faculty of Mathematics and Computer Science,\newline \hspace*{4mm}
Str. Mihail Kog\u{a}lniceanu nr. 1, 400084 Cluj-Napoca, Romania}
\email{tamas.laszlo@ubbcluj.ro}
\thanks{The author is partially supported by the  J\'anos Bolyai Research Scholarship of the Hungarian Academy of Sciences and the NKFIH Grant ``\'Elvonal (Frontier)'' KKP 144148.}

\title{On the canonical polynomial for links of elliptic singularities}

\begin{document}

\keywords{Normal surface singularities, links of singularities,
plumbing graphs, rational homology spheres, Poincar\'e series, elliptic singularities, elliptic sequence, canonical polynomial}

\subjclass[2010]{Primary. 32S05, 32S25, 32S50, 57K31
Secondary. 14B05, 14J80}

\begin{abstract}
The \textit{canonical polynomial} is an important output of the multivariable topological Poincar\'e series associated with a normal surface singularity. It can be considered as a multivariable polynomial generalization of the Seiberg--Witten invariant of the link. In the case of elliptic germs, 
another key topological invariant was considered, the  elliptic sequence, which mirrors the specific structure of the elliptic germs and guides several properties of them.

In this note  we study the relationship of these two objects. First of all, we describe  the structure of the exponents of the canonical polynomial and prove that they determine the elliptic sequence. For the converse problem, we consider an inductive setup of elliptic germs via natural extension of their graphs and compare the corresponding sets of exponents. This leads to the definition of a \textit{good extension} which can be characterized by an inclusion type formula for the corresponding canonical polynomials. This reflects in a compatible way  the `flag structure' of the elliptic sequence.
\end{abstract}

\maketitle


\linespread{1.2}


\pagestyle{myheadings} \markboth{{\normalsize T. L\'aszl\'o}} {{\normalsize On the canonical polynomial for links of elliptic singularities}}


\section{Introduction}

\subsection{} Classification of normal surface singularities and investigation of their analytic invariants via topological data is a classical research program in singularity theory.  In such a study 
certain  fundamental cycles (ie. special divisors supported on the exceptional set of a fixed resolution space) play a crucial role since they identify important analytical and topological  properties of the singularity.

The first results of this program can be attributed to Artin \cite{Artin,Artin66}, who defined topologically  a special cycle $Z_{min}$ (called the minimal cycle or Artin's fundamental cycle), and  characterized the rational singularities and their Hilbert-Samuel function via this cycle.  Laufer \cite{Laufer72,Lauferminell} extended Artin's results
 to the case of elliptic singularities: he introduced the class of  minimally elliptic singularities  as the Gorenstein singularities having geometric genus $1$. On one hand, Laufer proved that they can be characterized topologically and their Hilbert-Samuel function can be computed topologically. On the other hand, he noticed that for `more complicated' elliptic singularities (eg. having $p_g\geq 1$ without the Gorenstein condition, or Gorenstein elliptic singularities with $p_g\geq 2$) 
  all these fail  without extra assumptions. An important new step in the study of 
  (arbitrary) elliptic singularities was  done by S. S.-T.  Yau \cite{Yauell79,Yauell80}
   who defined the elliptic sequence for any elliptic germ (inspired by  constructions of Laufer).
  They constitute some kind of `topological skeleton' for the topological type of an elliptic singular germ. Motivated by the results of Laufer and Yau, N\'emethi \cite{Nell} showed  that for Gorenstein singularities with rational homology sphere link the so-called `Artin-Laufer' program of topological characterization of the geometric genus and of the Hilbert-Samuel function can be continued.

In the case of elliptic germs, the main topological invariant, which guides most of the topological and analytical properties of elliptic singularities is the {\em elliptic sequence}.  The first versions of the elliptic sequence were introduced by Laufer and Yau \cite{Yauell79,Yauell80}. It is a sequence of integral cycles.  Recently, Nagy and N\'emethi \cite{NNIII,NNell} constructed another version, which coincides with Yau's sequence in the numerically Gorenstein case, however it differs from it otherwise. In this second case it is a sequence of rational cycles and  it captures differently the failure of the numerically Gorenstein  property. In this article we will rely on this second  approach.

The aim of the present work is to show how the elliptic sequence  can be extracted from the exponents of certain topological-combinatorial Poincar\'e series 
(zeta-function) associated with the topology of the singularity. 
Such series can be defined for any surface singularity with rational homology sphere link. Several results of the last two decades show that they are extremely strong invariants of the singularities, 
not only that they code deep topological invariants of the link, but by their nature they constitute a key bridge with the analytical Poincar\'e series (associated with divisorial filtrations). For more details see the next subsection.

The motivation of the present work is two folded. On one hand, we believe this study gives a hint for defining such a `characteristic' sequences in the general case of normal surface singularities, as a generalization of the elliptic sequence (with all its benefits).  
On the other hand, it highlights (once again) by a new concrete way how the exponents of the Poincar\'e series can encode  key information.  The present connection was completely hidden before this work.

\subsection{}\label{sec:1.2} In order to provide some details regarding the results, we will introduce the following terminology.

Let $(X,o)$ be a normal surface singularity and we will assume that its link $M$ is a rational homology sphere ( $\mathbb{Q}HS^3$ for short). We fix a good resolution $\pi:\tilde X\to X$ of $X$. In particular,  the exceptional divisor $E:=\pi^{-1}(0)$ is a simple normal-crossing divisor. Let $\Gamma$ be the dual resolution graph with set of vertices $\calv$.

One considers the usual combinatorial package of the resolution (for details see sect.~\ref{ss:PGP}) and define the topological Poincar\'e series as follows. Let $\cup_{v\in\calv}E_v$ be the decomposition of $E$ into irreducible components,
$L=H_2(\tilde X, \Z)=\Z\langle E_v\rangle_v$ is the lattice associated with $\pi$( or with $\Gamma$) endowed with the negative definite intersection form $(E_v,E_w)_{v,w}$. $L'$ is the dual lattice identified with the
rational cycles $l'\in L\otimes \mathbb{Q}$ for which $(l',l)\in \Z$ for any $l\in L$.
Then $L'/L$ is the finite group $H:=H_1(M,\Z)$ and $[l']$ denotes the class of $l'\in L'$ in $H$. Given two cycles $l'_i=\sum_v l'_{i,v}E_v \in L'\subset L\otimes \Q$ ($i\in\{1,2\}$) one defines a natural partial ordering: $l'_1\geq l'_2$ if and only if $l'_{1,v}\geq l'_{2,v}$ for any $v\in \calv$.
We consider the (anti)dual basis $\{E^*_v\}$ of $L'$, see section \ref{sss:comb}.  Using these dual cycles one defines a `zeta-type' rational function $\prod_{v\in\calv} (1-\bt^{E^*_v})^{\delta_v-2}$, where $\delta_v$ is the valency of the vertex $v$ and $\bt^{l'}:=\prod_v t_v^{l'_v}$ if $l'=\sum_v l'_v E_v$. Its Taylor expansion at the origin $Z(\bt)=\sum_{l'\in L'}z(l')\bt^{l'}$ is called the {\it topological Poincar\'e series} associated with $\Gamma$. \vspace{0.3cm}

The theory of topological Poincar\'e series was started with the work of Campillo, Delgado and Gusein-Zade (see eg. \cite{cdg, CDGPs,CDGEq, CDGZ-equivariantPS} and the references therein) and N\'emethi \cite{NPS}, see also \cite{NN1}. In the last decade several developements have been made regarding its applications, cf. \cite{NN1,NJEMS,NCL,LN,LNNSurg,LSz,LSzdiv,LNNdual, deltaan, deltatop, delta3}.

The series $Z(\bt)$ decomposes into $\sum_{h\in H} Z_h(\bt)$, where $Z_h(\bt):=\sum_{[l']=h}z(l')\bt^{l'}$. Furthermore, for any $h\in H$ and subset $I\subset \calv$ we can define the reduced series $Z_h(\bt_I):=Z_h(\bt)|_{t_i=1,i\notin I}$. For each $Z_h(\bt_I)$  there exists a unique decomposition $Z_h(\bt_I)=P_{h,I}(\bt_I)+Z^{neg}_{h,I}(\bt_I)$ \cite{LSz,LSzdiv, LNNdual} guided by the following properties: $P_{h,I}(\bt_I)$ is a  Laurent polynomial with  exponents $l'$, where 
$l'\not\leq 0$ (with respect to the partial ordering), 
and $Z^{neg}_{h,I}(\bt_I)$ is a rational function of negative degree in all variables. In general, the sum of the coefficients $P_{h,I}({\mathbf 1})$ of $P_{h,I}$ is equal to the periodic constant of $Z_{h,I}$ (see \ref{ss:pctPS} ), measuring the `regularity' of the series. In fact, this associated
numerical invariant  in the case $I=\calv$ provides the normalized Seiberg--Witten invariants of the link associated with the corresponding $spin^c$-structure indexed by the group element $h\in H$, cf. \cite{Lim, NOSZ, NJEMS}. This invariant is realized for certain smaller sets $I\subset \calv$ as well, thanks to a reduction procedure, see \cite{LN,LNRed}.

In particular, associated with $h=0$ and  $I=\calv$ the `polynomial part' $P_0(\bt)$ of the above decomposition  is called the {\em canonical polynomial}.  (For simplicity, the index $I=\calv$ is always omitted from $P_{0,\calv}(\bt)$.)
This polynomial can be thought as the analog of the Alexander polynomial associated with knots/links in $S^3$. In fact, a similar procedure can be applied to the Poincar\'e series associated with a plane curve singularity as well, when the Alexander polynomial of the corresponding  knot/link appears in such decomposition, see eg. \cite{cdg} and \cite[Sect. 5.4]{LNNdual}.

\subsection{}
Now, let $(X,o)$ be an elliptic singularity with $\Q HS^3$ link. Note that by a result of Yau \cite{Yauell80} there are very few pathological cases for which the minimal and the minimal good resolution do  not coincide. For these cases the results of the present article can be adapted easily to the level of the minimal good resolution. Therefore, we will omit these cases from our discussion and always assume that the minimal resolution is good.

We consider the elliptic sequence defined by Nagy and N\'emethi \cite{NNIII} (NN-elliptic sequence for short). It is linked to a sequence of
subgraphs $\{B_j\}_{j=-1}^m$, where $B_{-1}:=\Gamma$ and $B_m$ is the support of the minimal elliptic cycle associated with the resolution, see \ref{ss:NNcons}. We say that the length of the sequence is $m+1$. One also associates to this sequence of subgraphs a sequence of special cycles $\{C_j\}_{j=-1}^m$ which play a crucial role in our study. They agree with the NN-elliptic sequence. 
One knows by \cite{NNell} (see also \cite{Nlat}) that the length of the elliptic sequence equals to the normalized canonical Seiberg--Witten invariant of the link of the singularity.  Hence, it equals to $P_0(\mathbf{1})$ as well. This result was one of the motivations
 of the present manuscript.

\vspace{0.3cm}

The overall goal of this work is to understand the role of the exponents of the canonical polynomial. In this context, first we determine the structure of the exponents of $P_0(\bt)$ and prove that they determine the elliptic sequence $\{B_j,C_j\}$. More precisely, one proves that the non-negative entries  of an exponent determine $B_{j+1}$ for some $j$ and the cycle formed from  these components is exactly the special cycle $C_j$, see Theorem \ref{thm:exp} and the discussion therein.

Conversely, we study in what extent can the exponents be reconstructed from the elliptic sequence. For that, we consider an extension (an inclusion) $\Gamma'\subset \Gamma$ of elliptic graphs and we study the projections of the exponents of $\Gamma$ to $\Gamma'$ (Theorem \ref{thm:dualmapinj}). On the other hand, there might be exponents of $\Gamma'$ which do not come from such  projections, that is, they can not be `extended' to $\Gamma$. Nevertheless, in Theorem \ref{thm:cc} we prove that whenever an exponent of $\Gamma'$ is extendable to $\Gamma$, the coefficient of its monomial in $P_0^{\Gamma'}$ equals to the sum of the coefficients in $P_0^{\Gamma}$ corresponding to its extensions. Furthermore, we construct a step by step algorithm which builds up all the exponents of $\Gamma$ from the elliptic sequence.

The study of extendability of the exponents leads to the definition of a `good extension' of elliptic graphs which can be characterized by a natural inclusion type formula for the corresponding cannonical polynomials, as presented in Theorem \ref{thm:mainrec}. This formula is in accordance with the flag structure $B_{m}\subset B_{m-1}\subset\dots\subset \Gamma$ provided by the elliptic sequence.

\subsection{} The structure of the article is as follows. Section \ref{s:prliminaries} presents some preliminaries regarding the topology of normal surface singularities, Seiberg--Witten invariants, Poincar\'e series and the canonical polynomial. Section \ref{s:ell} is a short review of elliptic singularities with a high emphasis on the NN-elliptic sequence. Section \ref{s:shape} explains how (the exponents of) the canonical polynomial determines the elliptic sequence. We illustrate our result on detailed examples as well. In Section \ref{s:5} we study the projection of the exponents to elliptic subgraphs.  Reversely, in Section \ref{s:6} we discuss extensions of elliptic graphs and their corresponding exponents. We define the notion of `good extension' which, in some sense, `behaves' nicely with respect to a natural recursivity of canonical polynomials in the elliptic case. Finally, Section \ref{s:7} proves that this natural recursive (surgery) formula for the corresponding canonical polynomials is valid exactly when we have a good extension.


\section{Preliminaries: invariants of normal surface singularities}\label{s:prliminaries}

\subsection{Plumbing/resolution graphs}\label{ss:PGP}
\subsubsection{}\label{sss:comb} We fix a connected plumbing graph $\Gamma$ whose associated intersection matrix is negative definite. We denote the corresponding oriented
plumbed 3--manifold by $M=M(\Gamma)$. Any such $M(\Gamma)$ is the link of a complex normal surface singularity
$(X,o)$, which has a resolution $\widetilde{X}\to X$ with
resolution graph $\Gamma$ (see e.g. \cite{Nfive}),
and the complex analytic smooth surface $\widetilde{X}$ as a smooth manifold is the plumbed 4--manifold associated with $\Gamma$, with boundary $\partial \widetilde{X} = M$.

In this article we always assume that $M$ is a \emph{rational homology sphere},
equivalently, $\Gamma$ is a tree with all genus decorations  zero.

We will use the notation $\mathcal{V}$ for the set of vertices of $\Gamma$, $\delta_v$ for the valency of a vertex $v$,
$\mathcal{N}$ for the set of nodes (vertices with $\delta_v\geq 3$), and
$\mathcal {E}$ for the set of  end--vertices (ie. vertices with $\delta_v=1$).


As in section \ref{sec:1.2}, we consider the lattice $L:=H_2(\widetilde{X},\mathbb{Z})=\mathbb{Z}\langle E_v\rangle$ endowed with the negative definite intersection form  $(\,,\,)$. We set the notation $e_v:=(E_v,E_v)$. One also considers the dual lattice  $L':={\rm Hom}_\bZ(L,\bZ)=\{l'\in L\otimes\bQ\,:\, (l',L)\in\bZ\}$ which  is generated over $\mathbb{Z}$ by the (anti)dual classes $\{E^*_v\}_{v\in\mathcal{V}}$ defined by $(E^{*}_{v},E_{w})=-\delta_{vw}$, the opposite of the Kronecker symbol. Then one has the inclusions  $L\subset L'\subset L\otimes \bQ=:L_{\bQ}$ and  the intersection form extends to $L_{\bQ}$. Furthermore, we get the finite group $H:= L'/L\simeq H_1(M,\mathbb{Z})$.

For a cycle $l'=\sum_vl'_vE_v$ we define its support as $|l'|:=\{v\in \calv \ : \ l'_v\neq 0\}$. Usin the partial ordering defined in section \ref{sec:1.2}, we say that $l'$ is an effective cycle if $l'\geq 0$. For any $l'_1,l'_2\in L'$ with $l'_i=\sum_v l'_{i,v}E_v$ ($i=\{1,2\}$) one can set  $\min\{l'_1,l'_2\}:= \sum_v\min\{l'_{1,v},l'_{2,v}\}E_v$ and
analogously $\min\{F\}$ for a finite subset $F\subset L'$. We will also use the notation $l'_1\prec l'_2$ if $l'_{1,v}< l'_{2,v}$ for all $v\in\calv$. (Note that $\not\geq$ differs from $\prec$.)

\subsubsection{\bf The Lipman cone and generalized Laufer algorithm}\label{ss:GLA}

The lattice $L'$ admits a partition parametrized by the group $H$, where for any $h\in H$ one sets $L'_h=\{\ell'\in L'\mid [\ell']=h\}\subset L'$. Note that $L'_0=L$.
Given a class $h\in H$ one can define the unique cycle
\begin{center}
$r_h:=\sum_v l'_v E_v\in L'_h$ with the property $0\leq l'_v<1$ for any $v$.
\end{center}

We define the \textit{rational Lipman cone} by
$$\cS_\mathbb{Q}:=\{\ell'\in L_\mathbb{Q} \ | \ (\ell',E_v)\leq 0 \ \mbox{for all} \ v\in \V\},$$
which is  generated over $\mathbb{Q}_{\geq 0}$ by $E^*_v$.
On also defines $\cS':=\cS_\mathbb{Q}\cap L'=\mathbb{Z}_{\geq 0}\langle E^*_v \rangle$ as the semigroup (monoid) of anti-nef  cycles in $L'$.
Since $\{E^*_v\}_v $ have positive entries, $\cS_\mathbb{Q}\setminus \{0\}$ is in the open first quadrant.

The \textit{Lipman cone} $\cS'$ also admits a natural equivariant partition  $\cS'_{h}=\cS'\cap L'_h$.
Furthermore, we have the following properties:
\begin{enumerate}
\item[(a)] if $l'_1,l'_2\in \cS'_{h}$ then  $l'_2-l'_1\in L$ and  $\min\{l'_1,l'_2\}\in \cS'_h$;
\item[(b)]
for any $s\in L'$ the set $\{l'\in \cS' \mid l'\not\geq s\}$ is finite;
\item[(c)] for any $h\in H$ there exists a \textit{unique minimal cycle} $s_h:=\min \{\cS'_{h}\}\in\cS'_h$.
\end{enumerate}
A given class $h\in H$ will be represented wither with $r_h$ or with $s_h$, depending on the situation. The minimal cycle is determined by the \textit{generalized Laufer algorithm} \cite{Laufer72},\cite[Lemma 7.4]{NOSZ} which  is as follows.

We fix $h\in H$. Then for any $l'\in L'_h$ there exists a unique
minimal element $s(l')\in \calS'_h$ satisfying $s(l')\geq l'$ which can be obtained by the following algorithm.
Set $x_0:=l'$. Then one constructs a computation sequence $\{x_i\}_i$ as follows.
If $x_i$ is already constructed and $x_i\not\in\cS'$ then there exits some $E_{v_i}$ such that $(x_i,E_{v_i})>0$.
Then take $x_{i+1}:=x_i+E_{v_i}$ (for some choice of $E_{v_i}$). The procedure stops after finitely many steps,
say at $x_t$, and necessarily $x_t=s(l')$.

In particular, if we start with $l'=E_v$ with an arbitrarily chosen $v\in \calv$ then $s(l')=\min\{\calS'\setminus \{0\}\}$ is the \emph{minimal (or Artin's fundamental) cycle $Z_{min}\in L$}, see \cite{Artin, Artin66, Laufer72}. In this case, the  algorithm is  the classical \emph{`Laufer algorithm targeting $Z_{min}$'}.

In fact $s(r_h)=s_h$ and $r_h\leq s_h$, however, in general $r_h\neq s_h$. Note that this fact does not contradict the minimality of $s_h$ in $\cS'_h$ since $r_h$ might not sit in $\cS'_h$.

\subsubsection{}
Let $Z_K\in L'$ be the anti-canonical cycle defined by  the adjunction formulae $(Z_K,E_v)=e_v+2$ for all $v$.
We say that the germ  $(X,o)$, or its topological type, is \emph{numerically Gorenstein} if $Z_K\in L$. Note that this property
is independent of the resolution, since it means that the line bundle $\Omega^2_{X\setminus \{o\}}$ of holomorphic 2--forms on $X\setminus \{o\}$ is topologically trivial, cf. \cite{Du}. Furthermore, one says that $(X,o)$ is \emph{Gorenstein} if $Z_K\in L$ and $\Omega^2_{\tX}$ is isomorphic to $ \calO_{\tX}(-Z_K)$ (or,
equivalently, if  $\Omega^2_{X\setminus \{o\}}$ is holomorphically trivial).

From the adjunction formulae one follows the following useful expression
 \begin{equation}\label{ZK}
 Z_K-E=\sum_{v\in\mathcal{V}}(\delta_v-2)E^*_v,
 \end{equation}
where we set $E:=\sum_{v\in\mathcal{V}}E_v$.

\subsection{Rational singularities}
We recall that a germ $(X,o)$ is rational if its geometric genus $p_g(X,o)=\dim H^1(\tX,\calO_{\tX})=0$. This vanishing  was characterized topologically by Artin \cite{Artin, Artin66} and Laufer \cite{Laufer72} as follows.

Let us define the Riemann--Roch function $\chi:L'\to \Q$, $\chi(l'):= -(l', l'-Z_K)/2$. Then, $(X,o)$ is rational if and only if $\chi(l)>0$ for any non-zero effective cycle $l$, or, equivalently, $\chi(Z_{min})=1$. These are also equivalent with the fact that along the Laufer algorithm targeting $Z_{min}$ the $\chi$ function is constant, or,  at every step in the algorithm one has  $(x_i,E_{v_i})=1$.  In fact, this implies that for all the vertices of a resolution graph of a rational singularity one has the following inequality:
\begin{equation*}
 -e_v\geq \delta_v-1.
\end{equation*}
For simplicity, we also say that a negative definite plumbing graph is a \emph{rational} if one of the above properties is satisfied.

\subsection{The topological Poincar\'e series}\label{ss:tPs}
Let $\bZ[[L']]$ be the $\bZ$-module generated by the monomials $\mathbf{t}^{l'}:=\prod_{v\in \calv}t_v^{l'_v}$,
where $l'=\sum_{v}l'_v E_v\in L'$. Then
the  \emph{multivariable topological Poincar\'e series} $Z(\bt)=\sum_{l'\in L'} z(l')\bt^{l'}
\in\bZ[[L']] $ is the Taylor expansion  at the origin of the rational function
\begin{equation}\label{eq:1.1}
\prod_{v\in \mathcal{V}} (1-\mathbf{t}^{E^*_v})^{\delta_v-2}.
\end{equation}

The expression (\ref{eq:1.1}) shows that $Z(\mathbf{t})$ is supported in the Lipman cone $\mathcal{S}'$. Moreover, the exponents are expressed as $l'=\sum_{v\in\caln\cup \cale} l^*_v E^*_v$ for some $l^*_v\in \mathbb{Z}_{\geq 0}$ satisfying $0\leq l^*_v\leq \delta_v-2$ if $v\in \caln$.  Hence, the coefficients are
\begin{equation}\label{eq:z}
z(l')={\prod_{v\in \caln}(-1)^{l^*_v}{{\delta_v-2}\choose {l^*_v}}}.
\end{equation}

The  series decomposes as $Z(\mathbf{t})=\sum_{h\in H}Z_h(\mathbf{t})$ where $Z_h(\mathbf{t})=\sum_{[l']=h}z
 (l')\bt^{l'}$, and we will use the following notation for the support:
\begin{equation*}
Supp(Z_h)=\{ l'\in \calS'_h \ : \ z(l')\neq 0\}\subset \calS'_h.
\end{equation*}
Note that, however the first exponent of $Z_0$ is always $l'=0$, in the case of $h\neq 0$ the minimal element $s_h$ of $\calS'_h$ might not be in $Supp(Z_h)$ as shown by the following example.
\begin{example}
Let $\Gamma$ be a star-shaped graph with $4$ vertices with the associated base elements $\{E_i\}_{i=1}^4$. Consider the following decorations: $E_1^2=-3$ for the central vertex and $E_i^2=-2$ for the other vertices. Then one computes that   $r_{[2E_1^*]}=(1/3)E_1+(2/3)E_2+(2/3)E_3+(2/3)E_4$ and $s_{[2E_1^*]}=r_{[2E_1^*]}+E_1=2E_1^*$. Hence, by the above discussion,  $s_{[2E_1^*]}$ is not in the support of $Z_{[2E_1^*]}$.
\end{example}

\subsection{\bf Counting functions and periodic constants of multivariable series}

In this section we will adopt the setting of section \ref{sss:comb} in a slightly more general context.
Namely, $L$ will be a lattice freely generated by base elements $\{E_v\}_{v\in \V}$,
$L' $ is an overlattice of the same rank (but not necessarily dual of $L$), and we consider the finite abelian group $H:=L'/L$.
One also defines the partial ordering on $L'$ as in section~\ref{sec:1.2}.

Consider a multivariable series $S(\bt)=\sum_{l'\in L'}a(l')\bt^{l'}\in \bZ[[L']]$,
let $Supp(S):=\{l'\in L' : a(l')\neq 0\}$ be its support  and we assume the following condition:
\begin{equation}\label{eq:finiteness}
\{l'\in Supp(S) \mid l'\not\geq x\} \ \ \mbox{is finite for any} \ x\in L'.
\end{equation}

For any subset of vertices $\emptyset\neq I \subset \calv$ we can define the projection of $L'$ via the map
${\mathrm{pr}}_I:L_{\bQ} \to \oplus_{v\in I} \bQ \langle E_v\rangle$ as $L'_I={\mathrm{pr}}_I(L')$, and one can consider multivariable series in the  $\bZ$-module  $\bZ[[L'_I]]$. For example, if $S(\bt)\in \bZ[[L']]$ then $S(\bt_{I}):=S(\bt)|_{t_v=1,v\notin I}$ is called the \emph{reduced series}, which is an element of $\bZ[[L'_I]]$.
In the sequel we use the notation $l'_I=l'|_I:= {\rm pr}_I(l')$ and $\bt^{l'}_I:=\bt^{l'}|_{t_v=1,v\notin I}$ for any $l'\in L'$.
Each coefficient $a_I(x)$ of $S(\bt_I)$ is obtained as a summation of certain coefficients $a(y)$ of $S(\bt)$, where $y$ runs over
$ \{l'\in Supp(S)\mid l'_I =x \}$
(which is finite by~\eqref{eq:finiteness}). Moreover, $S(\bt_I)$ satisfies the  finiteness property ~\eqref{eq:finiteness} too.

We can consider the decomposition $S(\bt)=\sum_h S_h(\bt)$,
where $S_h(\bt):=\sum_{[l']=h}a(l')\bt^{l'}$.  Here one has to emphasize that
the $H$-decomposition of the reduced series $S(\bt_{I})$ is not well defined.  That is,
 the reduced series $S_h(\bt)|_{t_v=1,v\notin I}$  cannot be recovered from $S(\bt_I)$ in general,
since the class of $l'$ cannot be recovered from $l'|_I$. Hence, the notation $S_h(\bt_I)$, defined as $S_h(\bt)|_{t_v=1,v\notin I}$,  is not ambiguous, but requires certain caution.


Associated with $S_h(\bt_I)$ the the \emph{counting function} of its coefficients is defined by
\begin{equation}
Q^{(S)}_{h,\I}: L'_\I\longrightarrow \bZ, \ \ \ \
x_\I\mapsto \sum_{l'_\I\ngeq x_\I} a(l'_\I).
\end{equation}
It is well defined by~\eqref{eq:finiteness}, and can be extended to $x\in L'$ via the projection  $L'\to L'_\I$ as  $Q^{(S)}_{h,\I}(x)=Q^{(S)}_{h,\I}(x_\I)$. The same definition provides the counting function $Q{(S(\bt_\I))}$ for $S(\bt_I)$ too, and this case we get $Q{(S(\bt_\I))}=\sum_{h\in H} Q^{(S)}_{h,\I}$.

Now, we fix $h \in H$ and for simplicity we set $\I:=\calv$.  Thus we look at only $S_h(\bt)$.
Let $\mathcal{K}\subset L'\otimes\mathbb{R}$ be a real closed cone whose affine closure is top dimensional.
Assume that there exist $l'_* \in \mathcal{K}$ and a finite index sublattice $\widetilde{L}$ of $L$ and a
quasi-polynomial $\qp{Q}^{{\mathcal K},(S)}_{h,\calv}(x)$
defined on $\widetilde{L}$ such that
\begin{equation}
\label{eq:qpol}
\qp{Q}_{h,\calv}^{{\mathcal K},(S)}(l) = Q_{h,\calv}^{(S)}(r_h+l)
\end{equation}
whenever $l\in (l'_* +\mathcal{K})\cap \widetilde{L}$ ($r_h\in L'$ is defined  as in~\ref{ss:GLA}).
Then we say that the counting function $Q_{h,\calv}^{(S)}$ (or just $S_h(\mathbf{t})$)
\emph{admits a quasi-polynomial} in $\mathcal{K}$, namely $\widetilde{L}\ni \l\mapsto
\qp{Q}_{h,\calv}^{\mathcal{K},(S)}(l)$.
In this case, we can define the \emph{periodic constant} (\cite{LN}) of $S_h(\bt)$ associated with $\mathcal{K}$ by
\begin{equation}\label{eq:PCDEF}
\mathrm{pc}^{\mathcal{K}}(S_h(\bt)) := \qp{Q}_{h,\calv}^{{\mathcal K},(S)}(0) \in \bZ.
\end{equation}
Note that $\mathrm{pc}^{\mathcal{K}}(S_h(\bt))$ is independent of the
choice of $\l'_*$ and of the finite index sublattice
$\widetilde L\subset L$.

Given any $\I\subset \calv$ the natural group homomorphism ${\rm pr}_I:L'\to L'_I$ preserves the lattices $L\to L_I$
and  induces a homomorphism $H\to H_I:=L'_I/L_I$. However, note that even if
$L'$ is the dual of $L$ associated with a form $(\,,\,)$,
$L_I'$ usually is not a dual lattice of $L_I$, it is just an overlattice. This fact motivates the general setup of the present subsection.
Nevertheless, one can define the periodic constant associated with the reduced series $S_h(\bt_I)$ too by exchanging
$\calv$ (resp. $\bt$, $r_h$) by $\I$ (resp. $\bt_I$, $(r_{h})_I$).

The concept of periodic constants for one-variable series was defined by N\'emethi and Okuma in ~\cite{Ok,NOk}, the multivariable case was clarified in \cite{LN}. For more details we invite the reader to consult with \cite{NPS,NJEMS,LSz}.

\subsection{The case of topological Poincar\'e series and Seiberg--Witten invariants}\label{ss:pctPS}

In this section we consider the case of topological Poincar\'e series $Z_h(\bt)$, and recall some results emphasizing how it encodes important topological invariant of the link.

The positivity of the dual base elements $E_v^*$ implies that for any $x \in L'$  $Z_h(\bt_I)$ ($\emptyset\neq I\subset \calv$) satisfies the finiteness condition (\ref{eq:finiteness}), hence the corresponding counting functions are well-defined. For simplicity, throughout the paper we will simplify the notation to $Q_{h,I}(x):=Q^{(Z)}_{h,I}(x)$ for the counting function of~$Z_h(\bt_I)$. In particular we write $Q_h(x):=Q^{(Z)}_{h,\calv}(x)$ if there is no danger of ambiguity.

By the result of \cite{NJEMS} we know that for any $\l'\in Z_K+ int(\calS')$ (where $int(\calS')=\mathbb{Z}_{>0}\langle E^*_v\rangle_{v\in \calv}$) the
counting function of $Z_h(\bt)$ has the following form
\begin{equation}\label{eq:SUMQ}
{Q}_{h}(l')=
\chi(l')-\chi(r_h)+\mathfrak{sw}_h^{norm}(M),
\end{equation}
where $\mathfrak{sw}_h^{norm}(M)$ denotes the {\it normalized Seiberg--Witten invariant} of the link
$M$ indexed by $h\in H$. (For more details see eg. \cite{NOSZ,NJEMS,LN,LNNSurg}.) In particular, for $h=0$, $\mathfrak{sw}_0^{norm}(M)$ is called the canonical Seiberg--Witten invariant of $M$.
Furthermore, one deduces that $Z_h(\bt)$ admits the (quasi-)polynomial
$\qp{Q}_{h}(l)=\chi(l+r_h)-\chi(r_h)
+\mathfrak{sw}_h^{norm}(M)
$
in the cone $\calS'_{\mathbb{R}}:=\calS'\otimes \mathbb{R}$ and
$$\mathrm{pc}^{\cS'_{\mathbb{R}}}(Z_h(\bt))=\mathfrak{sw}_h^{norm}(M).$$

\begin{remark}\label{rem:ratsw}
 (1) \  If $M$ is the link of a rational singularity then the Seiberg--Witten invariant can be expressed as  $\mathfrak{sw}^{norm}_h(M)=\chi(r_h)-\chi(s_h)$, cf. \cite{NPS,LN}. In particular, $\mathfrak{sw}^{norm}_0(M)=0$.

(2) \ We also emphasize  that the normalized Seiberg--Witten invariant appears as the `Euler characteristic' of several homology theories associated with the link, such as lattice homology and Heegaard Floer homology. This also supports the  vanishing of the canonical Seiberg--Witten invariant in case (1), since the simplest manifolds from the viewpoint of these homology theories are the negative definite L-spaces, or, equivalently, the links of rational surface singularities.
\end{remark}

For Seiberg--Witten invariants of links of normal surface singularities, there exists  a general surgery type formula in terms of the graph, developed in \cite{LNNSurg}. Here we state only the $h=0$ case which will be used in the sequel.

Let $I\subset \calv$ be an arbitrary non--empty subset of $\calv$ as before. Consider the subgraph spanned by the vertices  $\calv\setminus I$ (ie.  we take out from $\Gamma$ the vertices $I$ and their associated edges) and write it as the union of full
  connected subgraphs  $\cup_i\Gamma_i$. Denote by $M(\Gamma_i)$ the plumbed 3-manifold associated with $\Gamma_i$. Then the surgery formula is as follows:
  \begin{equation}\label{eq:surg}
  \mathfrak{sw}_0^{norm}(M)-\sum_i \ \mathfrak{sw}_0^{norm}(M(\Gamma_i))= {\rm pc}^{{\rm pr}_I(\calS_{\R})} (Z_0(\bt_{I})).
  \end{equation}

\subsection{Polynomials from the topological Poincar\'e series}\label{s:PP}

%
%

Assume again that we are in the situation of a negative definite plumbing graph and its associated
topological Poincar\'e series  $Z(\bt)$. Fix $h\in H$ and a subset $\emptyset\neq I\subset \calv$, and write the expression  $\bt_{I}^{Z_K-E}  Z_{[Z_K]-h}(\bt_I^{-1})$ as $\sum _{l'\in Z_K-E-\calS'}w(l')\bt_I^{l'}$, where
$[l']=h$ automatically whenever $w(l')\not=0$. Then we define the polynomial
\begin{equation}\label{eq:poldef}
P_{h,I}(\bt_I):=\sum_{l'|_{I}\nprec 0|_{I}} w(l')\bt_I^{l'}\end{equation} and also the series
$Z^{neg}_{h,I}(\bt_I):= \sum_{l'|_{I}\prec 0|_{I}} w(l')\bt_I^{l'}$.
Set also the notation $P_{h,I}({\mathbf 1}):=
P_{h,I}(\bt_I)|_{t_i=1, \forall \, i}$.
Then one has the following result.

\begin{theorem}[\cite{LNNdual}]\label{th:fneq}
(a) There is a unique decomposition
$Z_h(\bt_I)=P_{h,I}(\bt_I)+Z^{neg}_{h,I}(\bt_I)$  satisfying the following properties:
(i) $P_{h,I}(\bt_I)$ is a finite sum (Laurent polynomial) with non-negative exponents;
(ii) $Z^{neg}_{h,I}(\bt_I)$ is a rational function with negative degree in all variables.

(b) $P_{h,I}({\mathbf 1})=
Q_{[Z_K]-h,I}(Z_K-r_h)=
{\rm pc} ^{\pi_I(\calS'_{{\mathbb R}})}(Z_h(\bt_I)).$
\end{theorem}
In particular, for $I=\calv$ we get
$$P_{h}({\mathbf 1})=\mathfrak{sw}_h^{norm}(M).$$
(For simplicity, we will always omit  the index $\calv$ and simply write $P_{h}(\bt)$ for the polynomial associated with $Z_h$.)
On the other hand, in the case of $h=0$ part (b) of the above theorem specializes to
\begin{equation}\label{eq:pccf}
 {\rm pc} ^{\pi_I(\calS'_{{\mathbb R}})}(Z_0(\bt_I))=Q_{[Z_K],I}(Z_K) \ \ \mbox{for any } \ I\subset \calv.
\end{equation}

\begin{definition}
For $h=0$, $P_0(\bt)$ is called the \emph{canonical polynomial} associated with $\Gamma$.
\end{definition}

Sometimes is very convenient to consider the \emph{dual polynomials} given by truncation of the series $Z_{[Z_K]-h}(\bt_I)$ as follows.
\begin{equation}\label{eq:dualpol0}
\check{P}_{h,I}(\bt_I):=
\sum_{[l']=[Z_K]-h\atop l'|_{I}\nsucc (Z_K-E)|_{I}} z(l')\bt_I^{l'}.
\end{equation}
Then using the defining identity  (\ref{eq:poldef}) one has
\begin{equation}\label{eq:dual}
P_{h,I}(\bt_I)=\bt_I^{Z_K-E} \check{P}_{h,I}(\bt_I^{-1}).
\end{equation}
\vspace{0.5cm}

For more details regarding plumbing graphs, plumbed manifolds, Seiberg--Witten invariants and their relations with normal surface singularities see eg.  \cite{BN,EN,Nfive,NJEMS,Nic04,NN1,NOSZ,trieste};
for Poincar\'e series see also \cite{cdg, CDGPs,CDGEq, CDGZ-equivariantPS, CHR, NPS}.

\section{Elliptic singularities and the NN-elliptic sequence}\label{s:ell}

\subsection{Elliptic singularities}{\cite{Lauferminell,Wagreich}} \label{ss:ell}
We recall that a normal surface singularity $(X,o)$ is \emph{elliptic} if $\min _{l\in L_{>0}}\chi(l)=0$, or, equivalently, $\chi(Z_{min})=0$. In this case, we say that the dual resolution graph $\Gamma$ is an elliptic graph.
It is known that if we decrease the decorations (Euler numbers), or we take a full subgraph of an elliptic graph, then we get either elliptic or a rational graph.

One defines the minimally elliptic cycle $C$ \cite{Lauferminell} satisfying the conditions $\chi(C)=0$ and
$\chi(l)>0$ for any $0<l<C$. It is unique with these properties, and if $\chi(l)=0$ ($l\in L_{>0}:=\{l\in L: l> 0\}$) then necessarily $C\leq l$. In particular, $C\leq Z_{min}$.  Moreover, if $|C|$ is its support (cf. \ref{ss:PGP}), then the uniqueness implies that the connected components of  $\Gamma\setminus |C|$ are rational graphs. In particular, any connected subgraph of $\Gamma\setminus |C|$ is rational.

If we assume further, that $\Gamma$ is the minimal resolution, then $Z_K\in \calS'$. Hence, in the numerically Gorenstein case we have
$$C\leq Z_{min}\leq Z_K.$$
In particular, an elliptic germ which, on the minimal resolution, realizes the equalities $C=Z_{min}=Z_K$ is called \emph{minimally elliptic}, cf.  Laufer \cite{Lauferminell}. They are characterized by the Gorenstein and $p_g=1$ property, see \cite{Lauferminell,Nell,Nfive}. In fact for any elliptic singularity $|C|$ supports a resolution graph of a minimally elliptic singularity.

 \emph{In the sequel we always assume that the minimal resolution of the elliptic singularity $(X,o)$ is also good}. Hence, in this article an elliptic graph $\Gamma$ will be the dual graph of such a resolution.


Most of the properties of an elliptic singularity is guided by the \emph{elliptic sequence}. For Gorenstein singularities it was defined by Laufer, then a definition for any elliptic germ was given by Yau \cite{Yauell79,Yauell80}. Recently, Nagy and N\'emethi in \cite{NNIII} have constructed another version based on the unique minimal cycle $s_[Z_K]$. Note that all these versions coincide for numerically Gorenstein elliptic graphs.
More details and comparisons of the elliptic sequences can be found in \cite{Yauell79, Yauell80,Nell,Nfive, OkumaEll,NNIII,NNell}.

In this article, we will only recall and discuss the properties of the  sequence from \cite{NNIII}, and we will refer to it as the \emph{NN-elliptic sequence}.

\subsubsection{\bf Construction of the NN-elliptic sequence}\label{ss:NNcons} Let $\Gamma$ be an elliptic graph as in the previous section, ie. it is the minimal resolution graph of an elliptic singularity $(X,o)$ which is also good.

The elliptic sequence consists of a sequence  of cycles $\{Z_{B_j}\}_{j=-1}^m$, or a sequence of connected subgraphs $\{B_j\}_{j=-1}^m$, where $Z_{B_j}$ is a certain `minimal'
 cycle associated with $B_j$.  $\{B_j\}_{j=-1}^m$ are defined inductively as follows.

By the minimality of the resolution one has $Z_K\in \calS'$ and $Z_K\geq s_{[Z_K]}$. Moreover, since the graph is not rational, we get  $Z_K>s_{[Z_K]}$, cf. \cite[Lemma 2.1.4]{NNIII}. Therefore, as a `pre-step' of the construction one sets the following:
$$B_{-1}:=\Gamma, \ \ Z_{B_{-1}}:=s_{[Z_K]} \ \ \mbox{and}\ \ B_0:=|Z_K-s_{[Z_K]}|.$$
Then, one proves that if $Z_K\in L'\setminus L$ then $|C|\subset B_0\subsetneq \Gamma$, $B_0$ is connected and supports a numerically Gorenstein elliptic graph with anti-canonical cycle $Z_K-s_{[Z_K]}$. Moreover, any numerical Gorenstein subgraph is contained in $B_0$, cf. \cite[Prop. 3.4.4]{NNell}.

Now, let $Z_{B_0}$ be the Artin's minimal cycle of $B_0$. Then  $C\leq Z_{B_0}\leq Z_K-s_{Z_K}$. If $Z_{B_0}=Z_K-s_{[Z_K]}$ then we stop at index  $m=0$. Otherwise, the procedure continues and at the $j$-th step ($j\geq 1$) one defines $B_j:=|Z_K-s_{[Z_K]}-Z_{B_0}-\dots-Z_{B_{j-1}}|$ which will be a connected numerically Gorenstein elliptic graph with anti-canonical cycle $Z_K-s_{[Z_K]}-Z_{B_0}-\dots-Z_{B_{j-1}}$ and satisfies $|C|\subseteq B_j\subsetneq B_{j-1}$. Then, one sets $Z_{B_j}:=Z_{min}(B_j)$  for which we have $C\leq Z_{B_1}\leq Z_K-s_{[Z_K]}-Z_{B_0}-\dots-Z_{B_{j-1}}$.
After finite steps the coincidence of the minimal cycle and the canonical cycle on
 $B_m$ will stop the inductive procedure. This means that $B_m$ supports a minimally elliptic singularity with $Z_{B_m}=C$.

 We say that the NN-elliptic sequence $\{Z_{B_j}\}_{j=-1}^m$ has length $m+1$ with `pre-term' $Z_{B_{-1}}=s_{Z_K}\in L'_{[Z_K]}$.

\begin{remark}\label{rem:NGES}
In particular, if $\Gamma$ is a numerically Gorenstein elliptic graph then $s_{[Z_K]}=0$, hence $Z_{B_{-1}}=0$, $B_0=\Gamma$ and the algorithm starts with constructing the numerically Gorenstein subgraph $B_1$.
\end{remark}

\subsection{The cycles $C_j$ and $C'_j$}
Using the elliptic sequence, it is convenient to define the following cycles:
\begin{equation}\label{def:Ccycles}
C_j=\sum_{i=-1}^j Z_{B_i}\in L'_{[Z_K]} \ \ \mbox{and} \ \  C'_j=\sum_{i=j}^m Z_{B_i}\  \ \  (-1\leq j\leq m).
\end{equation}
In particular, $C_{-1}=s_{[Z_K]}$, $C'_m=C$ and $C_m=C'_{-1}=Z_K$.

We will rather use the sequence of cycles $\{C_j\}_{j=-1}^m$ instead of the minimal ones, since they hide a `geometric' universal property which will be crucial in this article.

First of all, in the next lemma we summarize the main properties of the elliptic sequence and its associated cycles. The properties from (a)-(e) can be found  in \cite{Nell,NNIII,NNell}. The observation from (f) is new and it will be used for our purposes.

\begin{lemma}\label{e211}

(a)\ $B_{-1}=\Gamma,\ B_0=|Z_K-\szk|,\ B_1=|Z_K-\szk-\zbz|,\ \ldots, \ B_m=|C|; $
each $B_j$ is connected and the inclusions $B_{j+1}\subset B_j$ are strict for any $j\geq 0$. For $j=-1$ one has $B_0\subsetneq B_{-1}$ whenever $Z_K\notin L$, otherwise $B_0=B_{-1}$.
Furthermore, $Z_{B_{-1}}=\szk\in L'_{[Z_K]}$ and the other cycles $\zbz\supset\zbe\supset \cdots\supset \zm=C$ are integral.

(b)\ $(E_v,\zj)=0$ for any $E_v\subset B_{j+1}$ and $-1\leq j\leq m$.
In particular, $(\zi,\zj)=(C_i,\zj)=0$ for all $-1\leq i<j\leq m$.

(c)\ $Z_K=\sum_{i=-1}^m\zi$.

(d)\ $(E_v,C'_j)=(E_v,Z_K)$ for any $E_v\subset |C'_j|$. In other words,
$C'_j=Z_K^{B_j}$, where $Z_K^{B_j}$ is the anti-canonical cycle of the graph $|C'_j|=B_j$. 
Furthermore, $\chi(Z_{B_j})=\chi(C_j)=\chi(C_j')=0$.

(e)\ $C_j\in \calS'$.

(f)\ $(C_j,E_v)=\begin{cases}
       0 &\text{if } v\in \calv(B_{j+1})\\
       e_v+2 &\text{if }  v\in \calv(\Gamma\setminus B_{j+1}) \text{ and } (E_{B_{j+1}},E_v)=0\\
       e_v+1 &\text{if }  v\in \calv(\Gamma\setminus B_{j+1}) \text{ and } (E_{B_{j+1}},E_v)=1\\
     \end{cases} $.

\end{lemma}

\begin{proof}
 (a)-(d) follows from \cite[2.11]{Nell}(numerically Gorenstein case) and \cite[Lemma 3.2.3]{NNIII}. The proof of (e) can be found in \cite[Lemma 3.3.1]{NNell}.

 (f) If $v\in \calv(B_{j+1})$ the statement follows from (b). Assume now that $v\in \calv(\Gamma\setminus B_{j+1})$. Note that we have $Z_K-C_j=Z_K^{B_{j+1}}$, hence $(C_j,E_v)=e_v+2-(Z_K^{B_{j+1}},E_v)$. Therefore, we get $(C_j,E_v)=e_v+2$ whenever $(E_{B_{j+1}},E_v)=0$.

 The case $(E_{B_{j+1}},E_v)=1$ will be  a consequence of Lemma \ref{lem:ext}, see Corollary \ref{cor:e211f}.
\end{proof}

\begin{remark}
Note that in general (eg. for non-numerically Gorenstein case) the cycle $s_{[Z_K]}\in L'_{[Z_K]}$ is determined by the generalized Laufer algorithm, cf. sect. \ref{ss:GLA}. However, in the elliptic case it is given by the explicit formula $s_{[Z_K]}=Z_K-Z_K^{B_0}$.
\end{remark}

The crucial fact regarding the cycles $C_j$ is that they are the only cycles in $\calS'_{[Z_K]}$ below the anti-canonical cycle.

\begin{lemma}{\cite[Lemma 3.3.1]{NNIII}}\label{lem:list} \\
Assume that $l'\in \calS'_{[Z_K]}$ and $l'\leq Z_K$. Then $l'\in \{C_{-1},C_0,\dots,C_m\}$.
\end{lemma}

This also means that the elliptic sequence $\{B_j\}$ has some universal property which can be formulated as follows:
\begin{cor}\label{rem:UNIVERSAL}
Any numerically Gorenstein connected subgraph of the elliptic graph $\Gamma$ must be one of  the $B_j$ for  $j\geq 0$.
\end{cor}

As a closure of this section, we recall the main result of \cite{NNell} which identifies the canonical Seiberg--Witten invariant of the link of an elliptic germ with the length of the elliptic sequence. Namely, using the previous notations, we have
\begin{equation}\label{th:SWmain}
\mathfrak{sw}^{norm}_{0}(M(\Gamma))=m+1.
\end{equation}

\section{The shape of the canonical polynomial for elliptic singularities}\label{s:shape}

We fix an elliptic graph $\Gamma$. Associated with $\Gamma$ we consider the canonical  polynomial $P_0:=\sum_{\ell\in L} w(\ell)\bt^\ell$, cf. sect. \ref{s:PP}. Define its {\em support} $Supp(P_0)\subset L$ as the finite set of exponents $\ell\in L$ for which $w(\ell)\neq 0$.

\subsection{The structure of the exponents}
\subsubsection{} For any exponent $\ell=\sum_{v\in\calv}\ell_v E_v\in Supp(P_0)$ we associate the following subset of vertices
\begin{equation}\label{eq:negV}
\calv^{<0}(\ell):=\{v\in \calv \ : \ \ell_v<0\}.
\end{equation}
Note that by (\ref{eq:poldef}) $\ell$ satisfies $\ell\nprec 0$, hence
\begin{equation}\label{propV<0}
\calv^{<0}(\ell)\subsetneq \calv=:\calv(B_{-1}) \mbox{ \ is a proper subset.}
\end{equation}
Since, by (\ref{eq:dual}), one knows that $\ell\in Z_K-E-\calS'_{[Z_K]}$, we consider the unique decomposition
\begin{align}\label{eq:expdec}
\ell=Z_K-E-l'-\sum_{v\in \calv^{<0}(l)} m_v E_v \mbox{ \ satisfying the following properties:} \\
(1) \ l'\leq Z_K, \quad \ \ \ (2) \ (Z_K-l')|_{\calv^{<0}(\ell)}=0|_{\calv^{<0}(\ell)} \quad \mbox{ \ and \ } \quad \ (3) \ m_v\in \Z_{\geq 0}. \nonumber
\end{align}
\begin{definition}
For any exponent $\ell\in Supp(P_0)$ the unique cycle $l'\in L'_{[Z_K]}$ from (\ref{eq:expdec}) will be called {\em the associated cycle of} $\ell$.

We will also set the notation $\cell:=l'+\sum_{v\in \calv^{<0}(\ell)} m_v E_v \in Supp(\check{P}_{0})$ and call $\cell$ the {\em dual exponent associated with} $\ell$, where $Supp(\check{P}_{0})$ denotes the set of the exponents of the dual polynomial $\check{P}_{0}(\bt)$ obtained by truncation of the series $Z_{[Z_K]}(\bt)$ (cf. (\ref{eq:dual}) and (\ref{eq:dualpol0})), ie. they are elements from $Supp(Z_{[Z_K]})$ and satisfy $\cell \ngeq Z_K$.
\end{definition}

\subsubsection{} In the next result we prove that in the case of an elliptic graph the set of associated cycles are exactly $\{C_j\}$ provided by the NN-elliptic sequence.

\begin{theorem}\label{thm:exp}
If $l'\in L'$ is the associated cycle of an exponent $\ell\in Supp(P_0)$, then
\begin{enumerate}
 \item $l'\in \{ C_{-1},C_0, \dots, C_{m-1} \}$;
 \item $\calv^{<0}(\ell)=\calv(\Gamma\setminus B_{j+1})$ for any $\ell\in Supp_j(P_0)$, where we denote by $Supp_j(P_0)\subset Supp(P_0)$  the subset of those exponents whose associated cycle is $C_j$.
\end{enumerate}
\end{theorem}

\begin{proof}

(1) Since  $[l']=[Z_K]$ and $l'\leq Z_K$ by construction, in order to apply Lemma \ref{lem:list} we only have to show that $l'\in \calS'$, ie. $(l',E_{v'})\leq 0$ for any $v'\in \calv$.

First of all, in the case $v'\notin \calv^{<0}(l)$ we have $(E_{v'},\sum_{v\in \calv^{<0}(\ell)}m_v E_v)\geq 0$. Then, together with the fact that $l'+\sum_{v\in \calv^{<0}(l)} m_v E_v \in \calS'$ one deduces the following inequalities
$$(l',E_{v'})\leq -(E_{v'},\sum_{v\in \calv^{<0}(\ell)}m_v E_v)\leq 0.$$

For $v'\in \calv^{<0}(\ell)$, first we decompose $(l',E_{v'})$ as
$$(l',E_{v'})=(l'|_{\calv\setminus\calv^{<0}(\ell)},E_{v'})+(l'|_{\calv^{<0}(l)},E_{v'}).$$
Then, (\ref{eq:expdec})(2) gives $(l'|_{\calv^{<0}(\ell)},E_{v'})=(Z_K|_{\calv^{<0}(\ell)},E_{v'})$. Moreover (\ref{eq:expdec})(1) and the assumption $v'\in \calv^{<0}(\ell)$ imply $(l'|_{\calv\setminus\calv^{<0}(\ell)},E_{v'})\leq (Z_K|_{\calv\setminus\calv^{<0}(\ell)},E_{v'})$.
Therefore, one gets $(l',E_{v'})\leq (Z_K,E_{v'})\leq 0$ where the last inequality is implied by the minimality of the resolution, cf. sect. \ref{ss:ell}.
This proves $l'\in \calS'$.

Finally, by applying Lemma \ref{lem:list} we get $l'\in \{ C_{-1},C_0, \dots, C_{m-1} \}$.

Note that $C_m$ does not appear in our list by the properness of the subset $\calv^{<0}(\ell)$, ie. there exists $v\in \calv\setminus\calv^{<0}(\ell)$ such that $l'_v<Z_{K,v}$ ($E_v$-coefficient of $Z_K$), see (\ref{propV<0}).

(2) Lemma \ref{e211}(d) says that for any $j\in\{-1,0,\dots,m-1\}$  $C'_{j+1}=Z_K-C_j$ is the cannonical cycle on its support $B_{j+1}$. Then, by using the definition (\ref{eq:negV}) of $\calv^{<0}(\ell)$ and the decomposition $\ell=Z_K-C_j-E-\sum_{v\in \calv^{<0}(\ell)} m_v E_v$ with its properties from (\ref{eq:expdec}), the result follows.
\end{proof}

\begin{corollary}\label{cor:sign}
For any $\ell\in Supp_j(P_0)$ ($j\in \{-1,\dots, m-1\}$) one has $\ell|_{B_{j+1}}\geq 0$ and $\ell|_{\Gamma\setminus B_{j+1}}\prec 0$.
\end{corollary}

We note that Theorem \ref{thm:exp} induces the following  decomposition:
$$Supp(P_0)=\sqcup_{j=-1}^{m-1} Supp_j(P_0).$$
Similarly, one writes $Supp(\check{P}_0)=\sqcup_{j=-1}^{m-1} Supp_j(\check{P}_0)$. On the other hand, the dual exponent associated with $\ell \in Supp_j(P_0)$ is expressed uniquely as
$$\check{\ell}=C_j+\sum_{v\in \calv(\Gamma\setminus B_{j+1})}m_v E_v$$ for some $m_v\geq 0$. Moreover, it satisfies $(\check{\ell}, E_v)=0$ for any $v\in \calv(B_{j+1})$. Indeed, if we assume  $(\check{\ell},E_v)\neq 0$ then the facts $\check{\ell}|_{B_{j+1}}=C_j|_{B_{j+1}}$ and $(C_j,E_v)=0$ for any $v\in \calv(B_{j+1})$ (Lemma \ref{e211}(f)), and the positivity of $m_v$ imply  $(\check{\ell},E_v)> 0$, which contradicts to $\check{\ell}\in \calS'$.

Regarding the coefficients $m_v$ we have the following observation.
\begin{prop}\label{prop:mv}
If $m_v\neq 0$ then $v\in \calv(\Gamma\setminus B_{j+1})$ is not adjacent to $B_{j+1}$.
\end{prop}
\begin{proof}
For every $v'\in \calv(B_{j+1})$ one has $(\cell,E_{v'})=\sum_{v\in \calv(\Gamma\setminus B_{j+1})}m_v \cdot(E_v,E_{v'})$ by the previous discussion. Moreover,  $(\cell,E_{v'})\leq 0$ since the series $Z_{[Z_K]}(\bt)$ is supported on $\calS'$, which implies the result.
\end{proof}

\begin{remark}
The geometric interpretation of the above property is as follows: assume that the elliptic graph $\Gamma$ is numerically Gorenstein. In this case, by \cite{Nell} for a Gorenstein singularity $(X,0)$  supported on $\Gamma$ there exists $f_j\in H^0(\widetilde X, \calO(-C_j))$, such that if $\text{div}_E(f_j)=C_j+l_j$ is the divisor of $f_j$ supported on $E$, then the support of $l_j$ does not contain any $E_v$ from $B_{j+1}$. Moreover, one can also conclude that it contains no $E_v$ which intersects $B_{j+1}$ nontrivially, cf. \cite[3.3.8]{NNIII}.
\end{remark}

In the next section we will prove that necessarily $Supp_j(P_0)\neq\emptyset$ for all $j$. This means  that the elliptic sequence $\{B_j\}$ and the cycles $\{C_j\}$ can be read off completely from the canonical polynomial $P_0(\bt)$.

At the same time, note that $C_j$ might not be in $Supp(Z_{[Z_K]})$, hence $Z_K-E-C_j$ is not necessarily an exponent of $P_0(\bt)$. For example, when the elliptic graph is numerically Gorenstein we have automatically that $Supp_{-1}(P_0)$ contains only $Z_K-E$ since the first exponent of $Z_0(\bt)$ is always $l'=0$. However, for non-numerically Gorenstein graphs this is not anymore true: the following example illustrates this behaviour by showing that $Z_K-E-s_{[Z_K]}\notin Supp_{-1}(P_0)$.

\begin{example}\label{ex:1}
\begin{figure}[h!]\label{fig1}
\begin{tikzpicture}[scale=.75]
\node (v1) at (-0.5,0) {};
\draw[fill] (-0.5,0) circle (0.1);
\node (v2) at (1,0) {};
\node (v3) at (1,-1.5) {};
\node (v4) at (2.5,0) {};
\node (v6) at (2.5,-1.5) {};
\node at (4,0) {};
\node (v7) at (5.5,0) {};
\node (v8) at (5.5,-1.5) {};
\node (v9) at (7,-1.5) {};
\node (v5) at (7,0) {};
\node (v11) at (8.5,-0.5) {};
\node (v10) at (8.5,0.5) {};
\node (v12) at (10,1) {};
\node (v13) at (10,0) {};
\draw[fill] (1,0) circle (0.1);
\draw[fill] (1,-1.5) circle (0.1);
\draw[fill] (2.5,0) circle (0.1);
\draw[fill] (2.5,-1.5) circle (0.1);
\draw[fill] (4,0) circle (0.1);
\draw[fill] (5.5,0) circle (0.1);
\draw[fill] (5.5,-1.5) circle (0.1);
\draw[fill] (7,-1.5) circle (0.1);
\draw[fill] (7,0) circle (0.1);
\draw[fill] (8.5,-0.5) circle (0.1);
\draw[fill] (8.5,0.5) circle (0.1);
\draw[fill] (10,1) circle (0.1);
\draw[fill] (10,0) circle (0.1);

\draw  (v1) edge (v2);
\draw  (v2) edge (v3);
\draw  (v2) edge (v4);
\draw  (v4) edge (v5);
\draw  (v4) edge (v6);
\draw  (v7) edge (v8);
\draw  (v7) edge (v9);
\draw  (v5) edge (v10);
\draw  (v5) edge (v11);
\draw  (v10) edge (v12);
\draw  (v10) edge (v13);
\node at (-0.5,0.5) {\small $-2$};
\node at (1,0.5) {\small $-2$};
\node at (1,-2) {\small $-2$};
\node at (2.5,0.5) {\small $-2$};
\node at (2.5,-2) {\small $-2$};
\node at (4,0.5) {\small $-4$};
\node at (5.5,0.5) {\small $-3$};
\node at (5.5,-2) {\small $-2$};
\node at (7,-2) {\small $-2$};
\node at (7,0.5) {\small $-3$};
\node at (8,0.5) {\small $-3$};
\node at (10.5,1) {\small $-2$};
\node at (10.5,0) {\small $-2$};
\node at (8,-0.5) {\small $-2$};
\node at (8.5,1) {\small $E_2$};
\node at (8.5,-1) {\small $E_1$};
\draw[dashed]  (-1,1.4) rectangle (10.8,-2.6);
\draw[dashed]  (-0.8,0.8) rectangle (4.6,-2.2);
\node at (4.2,-1.9) {\small $B_1$};
\draw[dashed]  (-0.9,1.1) rectangle (7.7,-2.4);
\node at (7.45,-2.2) {\small $B_0$};
\node at (10.45,-2.4) {\small $B_{-1}$};
\node at (9.5,1.1) {\small $E_{21}$};
\node at (9.5,-0.1) {\small $E_{22}$};
\end{tikzpicture}
\caption{A non-numerically Gorenstein graph}\label{fig:1}
\end{figure}
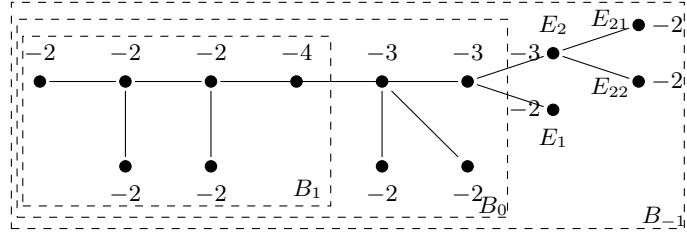
We consider the elliptic graph from Figure \ref{fig:1}. Then the cycles $Z_K$ and $s_{[Z_K]}$ are the followings.\\
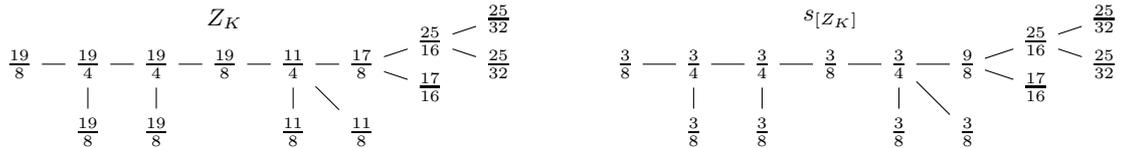
\begin{tikzpicture}[scale=.6]
\node (v1) at (-0.5,0) {\small $\frac{19}{8}$};
\node (v2) at (1,0) {\small $\frac{19}{4}$};
\node (v3) at (1,-1.5) {\small $\frac{19}{8}$};
\node (v4) at (2.5,0) {\small $\frac{19}{4}$};
\node (v6) at (2.5,-1.5) {\small $\frac{19}{8}$};
\node (v45) at (4,0) {\small $\frac{19}{8}$};
\node (v7) at (5.5,0) {\small $\frac{11}{4}$};
\node (v8) at (5.5,-1.5) {\small $\frac{11}{8}$};
\node (v9) at (7,-1.5) {\small $\frac{11}{8}$};
\node (v5) at (7,0) {\small $\frac{17}{8}$};
\node (v11) at (8.5,-0.5) {\small $\frac{17}{16}$};
\node (v10) at (8.5,0.5) {\small $\frac{25}{16}$};
\node (v12) at (10,1) {\small $\frac{25}{32}$};
\node (v13) at (10,0) {\small $\frac{25}{32}$};
\draw  (v1) edge (v2);
\draw  (v2) edge (v3);
\draw  (v2) edge (v4);
\draw  (v4) edge (v45);
\draw  (v45) edge (v7);
\draw  (v5) edge (v7);
\draw  (v4) edge (v6);
\draw  (v7) edge (v8);
\draw  (v7) edge (v9);
\draw  (v5) edge (v10);
\draw  (v5) edge (v11);
\draw  (v10) edge (v12);
\draw  (v10) edge (v13);
\node at (4,1) {\small $Z_K$};
\end{tikzpicture}
\hspace{1cm}
\begin{tikzpicture}[scale=.6]
\node (v1) at (-0.5,0) {\small $\frac{3}{8}$};
\node (v2) at (1,0) {\small $\frac{3}{4}$};
\node (v3) at (1,-1.5) {\small $\frac{3}{8}$};
\node (v4) at (2.5,0) {\small $\frac{3}{4}$};
\node (v6) at (2.5,-1.5) {\small $\frac{3}{8}$};
\node (v45) at (4,0) {\small $\frac{3}{8}$};
\node (v7) at (5.5,0) {\small $\frac{3}{4}$};
\node (v8) at (5.5,-1.5) {\small $\frac{3}{8}$};
\node (v9) at (7,-1.5) {\small $\frac{3}{8}$};
\node (v5) at (7,0) {\small $\frac{9}{8}$};
\node (v11) at (8.5,-0.5) {\small $\frac{17}{16}$};
\node (v10) at (8.5,0.5) {\small $\frac{25}{16}$};
\node (v12) at (10,1) {\small $\frac{25}{32}$};
\node (v13) at (10,0) {\small $\frac{25}{32}$};
\draw  (v1) edge (v2);
\draw  (v2) edge (v3);
\draw  (v2) edge (v4);
\draw  (v4) edge (v45);
\draw  (v45) edge (v7);
\draw  (v5) edge (v7);
\draw  (v4) edge (v6);
\draw  (v7) edge (v8);
\draw  (v7) edge (v9);
\draw  (v5) edge (v10);
\draw  (v5) edge (v11);
\draw  (v10) edge (v12);
\draw  (v10) edge (v13);
\node at (4,1) {\small $s_{[Z_K]}$};
\end{tikzpicture}
One can write $s_{\czk}=E_1^*+2E_2^*$. Since its $E_2^*$-coefficient is greater than $\delta_2-2=1$ (where $\delta_2$ is the valency of the vertex associated with $E_2$), $C_{-1}:=s_{\czk}$ does not appear in $Supp(Z_{\czk})$. Therefore $Z_K-E-s_{[Z_K]}\notin Supp_{-1}(P_0)$.
Nevertheless, by Theorem \ref{thm:exp} and Proposition \ref{prop:mv}, $s_{\czk}$ can be extended to a dual exponent $\check{\ell}=s_{\czk}+m_{21}E_{21}+m_{22}E_{22}$ associated with an exponent from $Supp_{-1}(P_0)$. Our claim (which will be proved later) is that the `only' possibilities are as follows: $\check{\ell}_1=s_{\czk}+E_{21}=E_1^*+E_2^*+2E_{21}^*$, $\check{\ell}_2=s_{\czk}+E_{22}=E_1^*+E_2^*+2E_{22}^*$, $\check{\ell}_3=s_{\czk}+2E_{21}=E_1^*+4E_{21}^*$, $\check{\ell}_4=s_{\czk}+E_{21}+E_{22}=E_1^*+2E_{21}^*+2E_{22}^*$, $\check{\ell}_5=s_{\czk}+2E_{22}=E_1^*+4E_{22}^*$.
\end{example}

\subsection{Counting coefficients of the canonical polynomial and its consequences}

Consider the coefficients $w(\ell)$ corresponding to the exponents $\ell\in Supp(P_0)$ in the canonical polynomial $P_0(\bt)$. These are the coefficients $z(\check{\ell})$ of the series $Z_{[Z_K]}(\bt)$.

Then, the first key observation of this section is that counting coefficients corresponding to the exponents in $Supp_j(P_0)$ is very special. Namely, we prove the following.

\begin{lemma}\label{lem:coefffunc}\
\begin{itemize}
\item[(a)] For every $j\in \{0,\dots,m\}$ and $-1\leq i< j$ one has the identity $$Q_{[Z_K],B_i\setminus B_j}(Z_K)=j.$$
In particular, $Q_{[Z_K],\Gamma\setminus B_j}(Z_K)=j$.
\item[(b)] $\sum_{\ell\in Supp_j(P_0)} w(\ell)=1$ for any $-1\leq j\leq m-1$.
\end{itemize}
\end{lemma}

\begin{proof}
(a) Following (\ref{eq:surg}) and (\ref{eq:pccf}), the counting function $Q_{[Z_K],B_i\setminus B_j}(Z_K)$ can be expressed via the surgery formula of the canonical Seiberg--Witten invariant as follows. Let $\cali=\calv(B_i\setminus B_j)$ and denote by $\{\Gamma_k\}_k$ the connected components of the subgraph $\Gamma\setminus (B_i\setminus B_j)$. Then, (\ref{eq:surg}) reads as
\begin{equation}
Q_{[Z_K],B_i\setminus B_j}(Z_K)= \mathfrak{sw}^{norm}_{0}(M(\Gamma))- \sum_k \mathfrak{sw}^{norm}_{0}(M(\Gamma_k)).
\end{equation}
One of the connected components is $B_j$ and the others are rational graphs by section \ref{ss:ell}.
Moreover, by (\ref{th:SWmain}) one has $\mathfrak{sw}^{norm}_{0}(M(\Gamma))=m+1$ and $\mathfrak{sw}^{norm}_{0}(M(B_j))=m+1-j$, while for the rational graphs we have always $\mathfrak{sw}^{norm}_{0}=0$, see Remark \ref{rem:ratsw}(1). Thus, one implies that $Q_{[Z_K],B_i\setminus B_j}(Z_K)=j$.

(b) By definition $Q_{[Z_K],\Gamma\setminus B_{j+2}}(Z_K)=\sum_{l'\in \calS'_{[Z_K]},\  l'|_{\Gamma\setminus B_{j+2}}\ngeq Z_K|_{\Gamma\setminus B_{j+2}}} z(l')$. For any $l'$ contributing in the previous sum one has $\ell=Z_K-E-l' \in Supp(P_0)$ since there exists $v\in \calv(\Gamma\setminus B_{j+2})$ with $l_v\geq 0$. Moreover, Corollary \ref{cor:sign} implies that
$$Q_{[Z_K],\Gamma\setminus B_{j+2}}(Z_K)=\sum_{t=-1}^{j}\sum_{\ell\in Supp_t(P_0)} w(\ell).$$
Hence, using (a) one can conclude with the following expression:
$$\sum_{\ell\in Supp_j(P_0)} w(\ell)=Q_{[Z_K],\Gamma\setminus B_{j+2}}(Z_K)-Q_{[Z_K],\Gamma\setminus B_{j+1}}(Z_K)=1.$$
\end{proof}
As we have already announced in the previous section, Lemma \ref{lem:coefffunc}(b) deduces immediately the following:
\begin{cor}
$Supp_{j}(P_0)\neq \emptyset$ for every $-1 \leq j\leq m-1$.
\end{cor}
This ensures that
\begin{equation}\label{eq:PtoC}
\mbox{\it the polynomial } P_0 \ \mbox{\it completely determines the elliptic sequence } \{B_j;C_j\}_j,
\end{equation}
except maybe $C_{-1}$, see Example \ref{ex:2} and Remark \ref{rem:nonnumGor} after therein.

Note that, however $P_0(\bt)$ might be complicated, we can get a simplified shape if we reduce it to the variables associated with the support of the minimally elliptic cycle $|C|$.

\begin{theorem}\label{thm:redminell}
For any elliptic graph one has
 $$P_0(\bt_{|C|})=\sum_{j=-1}^{m-1}\bt_{|C|}^{Z_K-E-C_j}.$$
\end{theorem}

\begin{proof}
Since $|C|\subseteq B_{j+1}$ for any $j\in\{-1,\dots, m-1\}$, Theorem \ref{thm:exp} implies that for any exponent $\ell\in Supp_j(P_0)$ one has $\ell|_{|C|}=(Z_K-E-C_j)|_{|C|}$. Moreover, the coefficient of $(Z_K-E-C_j)|_{|C|}$ in $P_0(\bt_{|C|})$ is given by the sum of all the coefficients of the exponents from $Supp_j(P_0)$, which equals $1$ by Lemma \ref{lem:coefffunc}(b).
\end{proof}

\begin{example}\label{ex:2} We consider the following graph.
\begin{figure}[h!]
\begin{tikzpicture}[scale=.75]
\node (v1) at (-0.5,0) {};
\draw[fill] (-0.5,0) circle (0.1);
\node (v2) at (1,0) {};
\node (v3) at (1,-1.5) {};
\node (v4) at (2.5,0) {};
\node (v6) at (2.5,-1.5) {};
\node at (4,0) {};
\node (v7) at (5.5,0) {};
\node (v8) at (5.5,-1.5) {};
\node (v9) at (7,-1.5) {};
\node (v5) at (7,0) {};

\draw[fill] (1,0) circle (0.1);
\draw[fill] (1,-1.5) circle (0.1);
\draw[fill] (2.5,0) circle (0.1);
\draw[fill] (2.5,-1.5) circle (0.1);
\draw[fill] (4,0) circle (0.1);
\draw[fill] (5.5,0) circle (0.1);
\draw[fill] (5.5,-1.5) circle (0.1);
\draw[fill] (7,-1.5) circle (0.1);
\draw[fill] (7,0) circle (0.1);

\draw  (v1) edge (v2);
\draw  (v2) edge (v3);
\draw  (v2) edge (v4);
\draw  (v4) edge (v5);
\draw  (v4) edge (v6);
\draw  (v7) edge (v8);
\draw  (v7) edge (v9);

\node at (-0.5,0.5) {\small $-2$};
\node at (1,0.5) {\small $-2$};
\node at (1,-2) {\small $-2$};
\node at (2.5,0.5) {\small $-2$};
\node at (2.5,-2) {\small $-2$};
\node at (4,0.5) {\small $-4$};
\node at (5.5,0.5) {\small $-3$};
\node at (5.5,-2) {\small $-2$};
\node at (7,-2) {\small $-2$};
\node at (7,0.5) {\small $-2$};

\node at (-0.5,-0.5) {\small $E_1$};
\node at (0.7,-0.5) {\small $E_2$};
\node at (0.5,-1.5) {\small $E_3$};
\node at (2.2,-0.5) {\small $E_4$};
\node at (2,-1.5) {\small $E_5$};
\node at (4,-0.5) {\small $E_6$};
\node at (5.2,-0.5) {\small $E_7$};
\node at (5,-1.5) {\small $E_8$};
\node at (7.5,-1.5) {\small $E_9$};
\node at (7.5,0) {\small $E_{10}$};
\end{tikzpicture}
\end{figure}
For simplicity, a cycle $l'=\sum_{i=1}^{10}l'_iE_i$ will be represented by the vector $(l'_1,\dots,l'_{10})$. The canonical polynomial  associated with this graph is as follows (calculations were performed by using \cite{Maple}):
\begin{align*}
P_0(\bt)=&\bt^{(1,3,1,3,1,1,1,0,0,0)} - 2\bt^{(0,1,0,1,0,0,-1,-1,-2,-1)}+\bt^{(0,1,0,1,0,0,-1,-1,-2,-1)}\\
&+\bt^{(0,1,0,1,0,0,-1,-1,-2,-2)}-2\bt^{(0,1,0,1,0,0,-1,-1,-1,-2)}+\bt^{(0,1,0,1,0,0,-1,-1,-1,-3)}\\
&+\bt^{(0,1,0,1,0,0,-1,-2,-2,-1)}+\bt^{(0,1,0,1,0,0,-1,-1,-1,-1)}+\bt^{(0,1,0,1,0,0,-1,-2,-1,-2)}\\
&+\bt^{(0,1,0,1,0,0,-1,-1,-3,-1)}+\bt^{(0,1,0,1,0,0,-1,-3,-1,-1)}-2\bt^{(0,1,0,1,0,0,-1,-2,-1,-1)}.
\end{align*}

One can see that there is an exponent $\ell=(1,3,1,3,1,1,1,0,0,0)$ such that $\ell\geq 0$. This implies that $\Gamma\setminus B_0=\emptyset$ by Corollary \ref{cor:sign} and the graph is numerically Gorenstein. Moreover, we get $Z_K=(2,4,2,4,2,2,2,1,1,1)$.

For all the other exponents one has $\ell|_{E_i}<0$ for every $i\in \{7,8,9,10\}$. Hence this is the index set for the vertices of $\Gamma\setminus B_1$, and all these exponents are belong to $Supp_0(P_0)$. The associated cycle of these exponents is $C_0$ which can be calculated by using the properties
\begin{center}
$Z_K|_{\Gamma\setminus B_1}=C_0|_{\Gamma\setminus B_1}$ and $C_0|_{B_1}=(Z_K-E-\ell)|_{B_1}$
\end{center}
for any $\ell\in Supp_0(P_0)$. Thus, we get $C_0=(1,2,1,2,1,1,2,1,1,1)$.

Since there are no `other type' of exponents we conclude that the subgraph $B_1$ is minimally elliptic and it ends the elliptic sequence.
\end{example}

\begin{remark}\label{rem:nonnumGor}
Similarly, in the non-numerically Gorenstein case all the subgraphs $\{B_j\}_{j=-1}^m$ and cycles $C_j$ except $C_{-1}=s_{[Z_K]}$ can be determined from the canonical polynomial. Only the difference $Z_K-s_{[Z_K]}$ is determined by the exponents appearing in $Supp_{-1}(P_0)$; $Z_K$ and $s_{[Z_K]}$ can not be seen directly from the exponents.
\end{remark}

Now, it is natural and interesting to ask the reverse of (\ref{eq:PtoC}):
\begin{center}
\emph{in what extent can the exponents (or the canonical polynomial) be constructed from the elliptic sequence and cycles $C_j$? }
\end{center}

Before studying this problem, we would like to emphasize that it is easy to construct graphs (even numerically Gorenstein ones) for which the elliptic sequence and the corresponding cycles are the same, but the polynomials are different. This behaviour is illustrated by the next example.

\begin{example}
Let us consider the following numerically Gorenstein elliptic graph $\Gamma$:
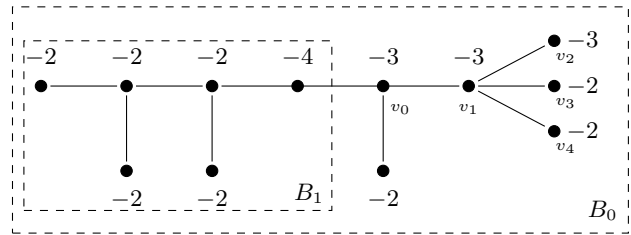
\begin{figure}[h!]
\begin{tikzpicture}[scale=.75]
\node (v1) at (-0.5,0) {};
\draw[fill] (-0.5,0) circle (0.1);
\node (v2) at (1,0) {};
\node (v3) at (1,-1.5) {};
\node (v4) at (2.5,0) {};
\node (v6) at (2.5,-1.5) {};
\node at (4,0) {};
\node (v7) at (5.5,0) {};
\node (v8) at (5.5,-1.5) {};
\node (v5) at (7,0) {};
\node (v11) at (8.5,-0.8) {};
\node (v10) at (8.5,0.8) {};
\node (v12) at (8.5,0) {};
\draw[fill] (1,0) circle (0.1);
\draw[fill] (1,-1.5) circle (0.1);
\draw[fill] (2.5,0) circle (0.1);
\draw[fill] (2.5,-1.5) circle (0.1);
\draw[fill] (4,0) circle (0.1);
\draw[fill] (5.5,0) circle (0.1);
\draw[fill] (5.5,-1.5) circle (0.1);
\draw[fill] (7,0) circle (0.1);
\draw[fill] (8.5,-0.8) circle (0.1);
\draw[fill] (8.5,0.8) circle (0.1);
\draw[fill] (8.5,0) circle (0.1);

\draw  (v1) edge (v2);
\draw  (v2) edge (v3);
\draw  (v2) edge (v4);
\draw  (v4) edge (v5);
\draw  (v4) edge (v6);
\draw  (v7) edge (v8);
\draw  (v5) edge (v10);
\draw  (v5) edge (v11);
\draw  (v5) edge (v12);
\node at (-0.5,0.5) {\small $-2$};
\node at (1,0.5) {\small $-2$};
\node at (1,-2) {\small $-2$};
\node at (2.5,0.5) {\small $-2$};
\node at (2.5,-2) {\small $-2$};
\node at (4,0.5) {\small $-4$};
\node at (5.5,0.5) {\small $-3$};
\node at (5.5,-2) {\small $-2$};
\node at (7,0.5) {\small $-3$};
\node at (9,0.8) {\small $-3$};
\node at (9,0) {\small $-2$};
\node at (9,-0.8) {\small $-2$};
\node at (7,-0.4) {\tiny $v_1$};
\draw[dashed]  (-1,1.4) rectangle (9.8,-2.6);
\draw[dashed]  (-0.8,0.8) rectangle (4.6,-2.2);
\node at (4.2,-1.9) {\small $B_1$};
\node at (9.35,-2.2) {\small $B_0$};
\node at (5.8,-0.4) {\tiny $v_{0}$};
\node at (8.7,0.5) {\tiny $v_{2}$};
\node at (8.7,-0.3) {\tiny $v_{3}$};
\node at (8.7,-1.1) {\tiny $v_{4}$};
\end{tikzpicture}
\end{figure}

The elliptic sequence consists of $B_0$ and $B_1$ as it is illustrated on the picture. One expresses $C_0=Z_{B_0}=2 E^*_{v_0}+E^*_{v_1}+E^*_{v_2}$ which is, in fact, not a dual exponent.  Hence a dual exponent has the form of $\cl=C_0+\sum_{i=1}^{4} m_{v_i} E_{v_i}\in Supp(Z_{0})$ with not all $m_{v_i}$ are zero. For example, we look at $\cl=C_0+E_{v_1}+E_{v_3}+E_{v_4}=E_{v_0}^*+2E^*_{v_1}+E^*_{v_3}+E^*_{v_4}$.

Now, we change slightly the graph $\Gamma$ by modifying the decoration of $v_2$ to $-2$. Then this new graph $\Gamma_{new}$  is elliptic with the same elliptic sequence.  However,  $\cl=C_0+E_{v_1}+E_{v_3}+E_{v_4}$ will not be  a dual exponent of $\Gamma_{new}$, since in this case $\cl\notin \calS'(\Gamma_{new})$.
\end{example}

\begin{remark}
 One can also construct families of elliptic graphs for which not only the elliptic sequences, but also the canonical polynomials agree.
\end{remark}

\section{Projections of the exponents}\label{s:5}

The forthcomming sections are devoted to study how the canonical polynomial can be built up from the elliptic sequence. For that, we consider an extension $\Gamma'\subset \Gamma$ of elliptic graphs and, first of all, we discuss the projection of exponenets of $\Gamma$ to $\Gamma'$.

\subsection{The index of the extension}\label{ss:index} In this section we will  introduce the notations $P_0^{\Gamma'}(\bt)$, $L(\Gamma')$, $L'(\Gamma')$, $l'^{\Gamma'}$ etc., where it is needed to specify the associated graph for the corresponding objects. We will also adopt all the settings from  the previous sections.

Associated with $\Gamma'\subset \Gamma$ we consider the projection/restriction map
\begin{align}
\begin{split}
\pi_{(\Gamma',\Gamma)}:  L(\Gamma) \ & \longrightarrow \ L(\Gamma')\\
 \ell \ \ \ \ \ \ & \longmapsto \ \ \ \ell|_{\Gamma'}.
\end{split}
\end{align}
In the sequel, we will study how does the map $\pi_{(\Gamma',\Gamma)}$ behave on the set of exponents of the canonical polynomial $P_0^{\Gamma}$ associated with $\Gamma$.

Let $\{B_j^{\Gamma}\}$ be the NN-elliptic sequence of $\Gamma$ together with the cycles $C_j^{\Gamma}\in \calS'_{[Z_K^{\Gamma}]}$, where $Z_K^{\Gamma}$ is the anti-canonical cycle associated with $\Gamma$. Since an exponent $\ell \in Supp_j(P_0^{\Gamma})$ can be written uniquely as $\ell=Z_K^{\Gamma}-E^{\Gamma}-C_j^{\Gamma}-\sum_{v\in \calv(\Gamma\setminus B_{j+1}^{\Gamma})}m_vE_v$ for some $m_v\geq 0$, its image equals with
\begin{equation}\label{eq:proj1}
\pi_{(\Gamma',\Gamma)}(\ell)=\ell|_{\Gamma'}=(Z_K^{\Gamma}-C_j^{\Gamma})|_{\Gamma'}-E^{\Gamma'}-\sum_{v\in \calv(\Gamma'\setminus B_{j+1}^{\Gamma})} m_v E_v.
\end{equation}

Consider also the NN-elliptic sequence $\{B_t^{\Gamma'},C_t^{\Gamma'}\}$ of $\Gamma'$. Then, by the universal property Corollary \ref{rem:UNIVERSAL} follows that any of the numerically Gorenstein subgraphs $B_t^{\Gamma'}$ for $t\geq 0$ must be equal to one of the $B_i^{\Gamma}$ for some $i\geq 0$. In particular, there exists an  $i_{(\Gamma',\Gamma)}\geq 0$, called the \emph{index of extension}, such that
\begin{equation}\label{eq:ext}
B_{\igam}^\Gamma \subseteq \Gamma' \ \ \mbox{and} \  \ B_{\igam-1}^\Gamma\nsubseteq \Gamma'.
\end{equation}
In this case $\{B_{-1}^{\Gamma'}=\Gamma', B_{\igam}^\Gamma, B_{\igam+1}^\Gamma\,\dots, B_{m}^\Gamma\}$ is exactly the elliptic sequence of $\Gamma'$. Therefore, we are interested to consider those exponents $\ell\in Supp_j(P_0^{\Gamma})$ for which $j\geq i_{(\Gamma',\Gamma)}-1$. In this case, the canonical cycle on $B_{j+1}^\Gamma\subset B_{i_{(\Gamma',\Gamma)}}^{\Gamma}$ is expressed as $Z_K^{B_{j+1}^\Gamma}=Z_K^\Gamma-C_j^\Gamma=Z_K^{\Gamma'}-C_{j-\igam}^{\Gamma'}$. Hence, according to (\ref{eq:proj1}), we get
\begin{equation}\label{eq:piform1}
\pi_{(\Gamma',\Gamma)}(\ell)=Z_K^{\Gamma'}-E^{\Gamma'}-C_{j-\igam}^{\Gamma'}-\sum_{v\in \calv(\Gamma'\setminus B_{j-\igam+1}^{\Gamma'})} m_v E_v,
\end{equation}
using also the identification $B_{j+1}^{\Gamma}=B_{j-\igam+1}^{\Gamma'}$.
\vspace{0.2cm}

In the followings, \emph{we will prove that with certain conditions, $\pi_{(\Gamma',\Gamma)}(\ell)$ will be an exponent of $\Gamma'$ belonging to $Supp_{j-\igam}(P_0^{\Gamma'})$.}

\subsection{The `dualized map' $\cpi_{(\Gamma',\Gamma)}$}

It will be useful to introduce the `dualized map' of $\pi_{(\Gamma',\Gamma)}$ defined by $\cpi_{(\Gamma',\Gamma)}:L'(\Gamma)_{[Z_K^{\Gamma}]}\to L'(\Gamma')_{[Z_K^{\Gamma'}]}$, $\cpi_{(\Gamma',\Gamma)}(l'):=Z_K^{\Gamma'}-E^{\Gamma'}-\pi_{(\Gamma',\Gamma)}(Z_K^{\Gamma}-E^{\Gamma}-l')$, and understand its behaviour on the set of dual exponents $Supp_j(\check{P}_0^{\Gamma})$ for $j\geq \igam -1$. For this reason, we make first the following preparation.

\subsubsection{}\label{sec:dualop}

In general, we consider any good resolution graph $\Gamma$ and fix a subset $I \subset \V$ of its vertices.
The set of vertices $\V \setminus I$ determines the connected full subgraphs $\{\Gamma_k\}_k$ with vertices
$\V(\Gamma_k)$, i.e. $\cup_k\V(\Gamma_k)=\V\setminus I$. We associate with any $\Gamma_k$ the lattices
$L(\Gamma_k)$ and $L'(\Gamma_k)$ as well, endowed with the corresponding intersection forms.
Then for each $k$ one considers the inclusion operator $j_k:L(\Gamma_k)\to L(\Gamma)$,
$E_v^{\Gamma_k}\mapsto E_v^{\Gamma}$, identifying naturally the corresponding $E$-base elements.  This preserves the intersection forms.

Let $j_{k}^*:L'(\Gamma)\to L'(\Gamma_k)$ be the dual operator, defined by
$j_{k}^*(E_{v}^{*\Gamma})=E_{v}^{*\Gamma_k}$ if $v\in\V(\Gamma_k)$, and $j_{k}^*(E_{v}^{*\Gamma})=0$ otherwise.
Note that $j^*_{k}(E_v^\Gamma)=E_v^{\Gamma_k}$ for any $v\in \V(\Gamma_k)$.
Futhermore, we have the projection formula
\begin{equation}\label{eq:projform}
(j^*_{k}(l'), \ell)_{\Gamma_k}=(l',j_{k}(\ell))_{\Gamma}
\end{equation}
for any $\l'\in L'(\Gamma)$ and $\ell\in L(\Gamma_k)$,
which also implies that
\begin{equation}\label{eq:proj}
j^*_{k}(Z_K^\Gamma)=Z_K^{\Gamma_k}.
\end{equation}

Similarly, this identification holds  for the minimal cycles $s_h$ as well:

\begin{lemma}[\cite{LSz}]\label{lem:projs_h}
For any $h\in H$ one has $j^*_{k}(s_{h}^\Gamma)=s_{[j^*_{k}(s_{h}^\Gamma)]}^{\Gamma_k}\in L'(\Gamma_k)$.
\end{lemma}

\subsubsection{\bf Description of $\cpi_{(\Gamma',\Gamma)}|_{Supp_j(\check{P}_0^{\Gamma})}$}

Let $\cell=C_j^\Gamma+\sum_{v\in \calv(\Gamma\setminus B_{j+1}^\Gamma)}m_v E_v\in Supp_j(\check{P}_0^\Gamma)$ be a dual exponent for some $j\geq \igam -1$. Then, by (\ref{eq:piform1}) one has
\begin{equation}\label{eq:dualpi}
\cpi_{(\Gamma',\Gamma)}(\cell)=C_{j-\igam}^{\Gamma'}+\sum_{v\in \calv(\Gamma'\setminus B_{j-\igam+1}^{\Gamma'})} m_v E_v.
\end{equation}
On the other hand, one proves that  the dual operator $j^*_{\Gamma'}$ gives the concrete relationship between the cycles $C_j^\Gamma$ and $C_{j-\igam}^{\Gamma'}$ as follows.

\begin{lemma}\label{lem:projC}
$j^*_{\Gamma'}(C_j^\Gamma)=C_{j-\igam}^{\Gamma'}$ for any $j\geq \igam-1$.
\end{lemma}

\begin{proof}
Since $Z_K^\Gamma-C_j^\Gamma=Z_K^{B_{j+1}^\Gamma}$ is supported on $B_{j+1}^\Gamma\subset B_{\igam}^\Gamma\subset \Gamma'$, it follows that $j^*_{\Gamma'}(Z_K^{B_{j+1}^\Gamma})=Z_K^{B_{j+1}^\Gamma}=j^*_{\Gamma'}(Z_K^{\Gamma})-j^*_{\Gamma'}(C_j^\Gamma)$. Furthermore, by (\ref{eq:proj}) we get $j^*_{\Gamma'}(C_j^\Gamma)=Z_K^{\Gamma'}-Z_K^{B_{j+1}^\Gamma}$. This, composed with the identity $Z_K^{B_{j+1}^\Gamma}=Z_K^{\Gamma'}-C_{j-\igam}^{\Gamma'}$, implies that  $j^*_{\Gamma'}(C_j^\Gamma)=C_{j-\igam}^{\Gamma'}$.
\end{proof}

Next, we define the following subsets of vertices: $\partial_\Gamma(\Gamma')$ is the set of vertices of $\Gamma'$ which are connected by an edge to $\Gamma\setminus \Gamma'$ in $\Gamma$. Similarly, $\partial_{\Gamma'}(\Gamma)$ denotes the set of  vertices $v\in \calv(\Gamma\setminus \Gamma')$ which are connected to $\Gamma'$. Since both $\Gamma'$ and $\Gamma$ are connected and trees, $\partial_{\Gamma'}(\Gamma)$ can be decomposed into $ \sqcup_{v\in \partial_{\Gamma}(\Gamma')} \partial_{\Gamma'}(\Gamma)_v$, where  $\partial_{\Gamma'}(\Gamma)_v$ contains all the neighbours of $v$ in $\Gamma\setminus \Gamma'$.

By \ref{sec:dualop} we know that $j^*_{\Gamma'}(E_u)=0$ for any $u\in \calv(\Gamma\setminus \Gamma')\setminus \partial_{\Gamma'}(\Gamma)$, while $j^*_{\Gamma'}(E_u)=-E^*_v$ for any $u\in \partial_{\Gamma'}(\Gamma)_v$ and $v\in \partial_{\Gamma}(\Gamma')$. Therefore, a straightforward calculation implies that for any $\cell=C_j^\Gamma+\sum_{v\in \calv(\Gamma\setminus B_{j+1}^\Gamma)}m_v E_v$ we have
\begin{equation}\label{eq:dualpiform}
\cpi_{(\Gamma',\Gamma)}(\cell)=j^*_{\Gamma'}(\cell)+\sum_{v\in \partial_{\Gamma}(\Gamma')}\big(\sum_{u\in \partial_{\Gamma'}(\Gamma)_v} m_u\big)\cdot E^*_v.
\end{equation}
This description of $\cpi_{(\Gamma',\Gamma)}$ and the projection formula (\ref{eq:projform}) immediately gives us the following identities.

\begin{lemma}\label{lem:cpi} For any $j\geq \igam -1$ and $\cell \in Supp_{j}(\check{P}^\Gamma_0)$ one has \\
(1) \ $(\cpi_{(\Gamma',\Gamma)}(\cell),E_v)_{\Gamma'}=(\cell,E_v)_{\Gamma}$ \ \ for $v\in \calv(\Gamma')\setminus \partial_{\Gamma}(\Gamma')$ and \\
(2) \ $(\cpi_{(\Gamma',\Gamma)}(\cell),E_v)_{\Gamma'}=(\cell,E_v)_{\Gamma}-\sum_{u\in \partial_{\Gamma'}(\Gamma)_v} m_u$ \ \ for every $v\in \partial_{\Gamma}(\Gamma')$.
\end{lemma}
Now, we can prove the main result of this section.

\begin{theorem}\label{thm:dualmapinj}
 If every $v\in \partial_{\Gamma}(\Gamma')$ is an end-vertex of $\Gamma'$, then $\pi_{(\Gamma',\Gamma)}(\ell) \in Supp_{j-\igam}(P_0^{\Gamma'})$ whenever $\ell \in Supp_j(P_0^{\Gamma})$ ($j\geq \igam-1$).
\end{theorem}

\begin{proof}
Let $\ell \in Supp_j(P_0^{\Gamma})$ ( for some $j\geq \igam-1$) be an exponent of $\Gamma$ and $\cell=C_j^\Gamma+\sum_{v\in \calv(\Gamma\setminus B_{j+1}^\Gamma)}m_v E_v$ be its dual exponent. We have to show that $\cpi_{(\Gamma',\Gamma)}(\cell)$ is a dual exponent of $\Gamma'$, that is an exponent of $Z_{[Z_K^{\Gamma'}]}$ and satisfies $\cpi_{(\Gamma',\Gamma)}(\cell)\ngeq Z_K^{\Gamma'}$, where $Z_{[Z_K^{\Gamma'}]}$ is the $[Z_K^{\Gamma'}]$-part of the topological Poincar\'e series associated with $\Gamma'$.

The latter is immediate from the identities (\ref{eq:dualpi}) and $Z_K^{\Gamma'}-C_{j-\igam}^{\Gamma'}=Z_K^{B_{j-\igam+1}^{\Gamma'}}$. On the other hand, we know that $\cpi_{(\Gamma',\Gamma)}(\cell) \in Supp(Z_{[Z_K^{\Gamma'}]})$ if and only if $\cpi_{(\Gamma',\Gamma)}(\cell)\in \calS'_{[Z_K^{\Gamma'}]}$ and for every node $v\in \calv(\Gamma')$ of $\Gamma'$ (ie. $\delta^{\Gamma'}_v \geq 3$) the inequality $-(\cpi_{(\Gamma',\Gamma)}(\cell),E_v)_{\Gamma'}\leq \delta^{\Gamma'}_v-2$ is satisfied.
By assumption one has $\cell \in \calS'_{[Z_K^{\Gamma}]}$ and $-(\cell,E_v)\leq \delta_v^\Gamma-2$ for every $v\in \calv(\Gamma)$ with $\delta_v^{\Gamma}\geq 3$. Therefore, Lemma \ref{lem:cpi} implies that $(\cpi_{(\Gamma',\Gamma)}(\cell),E_v)_{\Gamma'}\leq (\cell,E_v)_{\Gamma}\leq 0$ for any $v\in \calv(\Gamma')$, hence $\cpi_{(\Gamma',\Gamma)}(\cell)\in \calS'_{[Z_K^{\Gamma'}]}$. Moreover, by the assumption on  $\partial_{\Gamma}(\Gamma')$ and Lemma \ref{lem:cpi}(1), for every  node $v$ of $\Gamma'$ we have $-(\cpi_{(\Gamma',\Gamma)}(\cell),E_v)_{\Gamma'}=-(\cell,E_v)_{\Gamma}\leq \delta_v^{\Gamma'}-2$, since in this case $\delta_v^{\Gamma'}=\delta_v^{\Gamma}$.
\end{proof}

\section{Extensions of elliptic graphs and their exponents}\label{s:6}

In this part we consider the converse of the previous section and study how the exponents of $\Gamma'$ extend to the exponents of $\Gamma$.

\subsection{Small extensions of elliptic graphs}
In the following lemma we recall and  slightly extend the statement of \cite[Lemma 3.2.7]{NNIII}.

\begin{lemma}\label{lem:ext}
Let $\Gamma$ be an elliptic graph and $\{B_j\}_{j=-1}^m$ the NN-elliptic sequence. Assume that $\Gamma$ can be extended to a new elliptic graph $\Gamma_{new}$ by attaching new vertices $v_1,\dots,v_s$ to a fixed vertex $v_0\in \mathcal{V}(\Gamma)$. Then
\begin{itemize}
 \item[(a)]   $v_0\in \calv(B_{-1}\setminus B_1)$,
 \item[(b)]   $E_{v_0}$-multiplicity of $Z_{min}^\Gamma$ is $1$ and
 \item[(c)]  $(Z_{min}^\Gamma,E_{v_0})_{\Gamma}\leq 1-s$.
\end{itemize}
Furthermore, if $v_0\in \calv(B_0\setminus B_1)$ then $v_0$ is an end-vertex of $B_0$.
\end{lemma}

\begin{proof}
Suppose that $v_0\in \calv(B_1)$. Then the $E_{v_0}$-multiplicity of the cycle $Z_{B_0}+Z_{B_1}$, denoted by $m_{v_0}(Z_{B_0}+Z_{B_1})$,  is at least $2$ and $\chi_{\Gamma}(Z_{B_0}+Z_{B_1})=0$  implied by Lemma \ref{e211}. Hence, $\chi_{\Gamma_{new}}(Z_{B_0}+Z_{B_1}+E_{v_i})<0$ for all $i\in\{1,\dots,s\}$ which contradicts to the ellipticity of $\Gamma_{new}$. Thus,  $v_0\in \calv(B_{-1}\setminus B_1)$.

Consider the Laufer computation sequence $\{z_i\}_{i=0}^n$ of $Z_{min}^\Gamma$ in $\Gamma$ (ie. $z_n=Z_{min}^\Gamma$). This can be extended in $\Gamma_{new}$ to a Laufer's computation sequence (\dag) $z_0,\dots,Z_{min}^\Gamma,z_{n+1}, \dots, Z_{min}^{\Gamma_{new}}$.  Since $(Z_{min}^{\Gamma},E_v)_{\Gamma_{new}}=(Z_{min}^{\Gamma},E_v)_{\Gamma}\leq 0$ for any $v\in \calv(\Gamma)$ and $(Z_{min}^{\Gamma},E_{v_i})_{\Gamma_{new}}=m_{v_0}(Z_{min}^{\Gamma})>0$, part of the (\dag) computation sequence can be chosen in such a way that
\begin{equation}
z_{n+t}=Z_{min}^{\Gamma}+\sum_{j=1}^t E_{v_j} \ \ \mbox{for}\ 1\leq t\leq s.
\end{equation}
On the other hand, we know that $\chi_{\Gamma_{new}}(Z_{min}^{\Gamma})=\chi_{\Gamma_{new}}(Z_{min}^{\Gamma_{new}})=0$ and $\chi_{\Gamma_{new}}$ is decreasing along the sequence (\dag).  Hence for any $t\geq n$ we must have $(z_t,E_{v(t)})_{\Gamma_{new}}=1$ where $z_{t+1}=z_t+E_{v(t)}$. In particular,  $m_{v_0}(Z_{min}^{\Gamma})=1$ which proves (b). We get the inequality $(z_{n+s},E_{v(n+s)})_{\Gamma_{new}}\leq 1$ too.

Moreover, if the new decorations associated with $v_i$ are $e_i$, then (b) shows that for any $i\in \{1,\dots,s\}$ $(z_{n+s},E_{v_i})_{\Gamma_{new}}=e_{v_i}+1$, which is stricly negative. Therefore, the only candidate for $E_{v(n+s)}$ would be $E_{v_0}$ and we get $(z_{n+s},E_{v_0})_{\Gamma_{new}}\leq 1$, or, equivalently, $(Z_{min}^{\Gamma},E_{v_0})_{\Gamma}\leq 1-s$, thereby proving (c).

Now, assume $v_0\in \calv(B_0\setminus B_1)$.  Then $(Z_{B_0},E_{v_i})_{\Gamma_{new}}=m_{v_0}(Z_{B_0})\geq 1$ for any $i\in\{1,\dots,s\}$. Moreover, the facts $\chi_{\Gamma_{new}}(Z_{B_0})=0$ and $\chi_{\Gamma_{new}}(Z_{B_0}+E_{v_i})\geq 0$ imply $m_{v_0}(Z_{B_0})=1$, which gives also the equality $m_{v_0}(Z_K^{B_0})=1$ since $v_0\in \calv(B_0\setminus B_1)$. This shows that $v_0$ must be an end-vertex of $B_0$. Indeed, by the adjunction formula $v_0$ is either an end-vertex of $B_0$ or it has two neighbours in $B_0$ both with multiplicity $1$ in $Z_K^{B_0}$. But the last case would generate by repeating the argument an infinite string, all with multiplicity one, which cannot happen.
\end{proof}

\begin{definition}
For the ease of terminology,  an extension of elliptic graphs with properties as in Lemma \ref{lem:ext} will be called a \emph{small} extension.
\end{definition}

\begin{cor}[Part of Lemma \ref{e211} (f)]\label{cor:e211f}
Let $\Gamma$ be an elliptic graph, $\{B_j\}_{j=-1}^m$ its associated $NN$-elliptic sequence and $C_j$ the corresponding cycles. Then in the case $v\in \calv(\Gamma\setminus B_{j+1})$ and $(E_{B_{j+1}},E_v)=1$ one has $(C_j,E_v)=e_v+1$.
\end{cor}

\begin{proof}
As we have seen in the proof of Lemma \ref{e211} one has the identity $(C_j,E_v)=e_v+2-(Z_K^{B_{j+1}},E_v)$. By assumption, the numerically Gorenstein elliptic graph $B_{j+1}$ has a small extension by attaching to it the vertex $v$. We denote the neighbour of $v$ in $B_{j+1}$ by $v_0$. Then, by the end of the proof of Lemma \ref{lem:ext} we know that  $(Z_K^{B_{j+1}},E_v)=m_{v_0}(Z_K^{B_{j+1}})=1$.
\end{proof}

\subsection{Extensions of the (dual) exponents}

\subsubsection{}\label{ss:extcompseq} We consider a small extension $\Gamma\subset \Gamma_{new}$ as in Lemma \ref{lem:ext}. (Note that we have $\partial_{\Gamma_{new}}(\Gamma)=\{v_0\}$.) Then the main question guiding this section is the following:
\begin{center}\vspace{0.2cm}
{\it how can an exponent of $\Gamma$ be `extended' to an exponent of $\Gamma_{new}$?}
\end{center}
\vspace{0.2cm}

\underline{Now assume \textbf{$v_0$ is an end-vertex of $\Gamma$}. }
\vspace{0.2cm}

By Theorem \ref{thm:dualmapinj}  the (restriction of the) dualized map  $\cpi_{(\Gamma,\Gamma_{new})}: Supp(\check{P}_0^{\Gamma_{new}})\to Supp(\check{P}_0^{\Gamma})$ satisfies the property
$$\cpi_{(\Gamma,\Gamma_{new})}(Supp_j(\check{P}_0^{\Gamma_{new}}))\subset Supp_{j-i(\Gamma,\Gamma_{new})}(\check{P}_0^{\Gamma}).$$

Let $\cell=C^\Gamma_j+\sum_{\Gamma\setminus B^\Gamma_{j+1}} m_v E_v \in Supp_j(\check{P}_0^{\Gamma})$ be a dual exponent of $\Gamma$ and assume that $\cpi_{(\Gamma,\Gamma_{new})}^{-1}(\cell)\neq \emptyset$. Thus, there exist dual cycles $\cell_{new} \in \cpi_{(\Gamma,\Gamma_{new})}^{-1}(\cell) \subset Supp_{j+i(\Gamma,\Gamma_{new})}(\check{P}_0^{\Gamma_{new}})$ of $\Gamma_{new}$ which can be written as
$$\cell_{new}=C^{\Gamma_{new}}_{j+i(\Gamma,\Gamma_{new})}+\sum_{\Gamma\setminus B^{\Gamma}_{j+1}}m_v E_v + \sum_{i=1}^s m_i E_{v_i}$$
for some $m_i\geq 0$. In the following, we will determine all these possible extensions $\cell_{new}\in \cpi_{(\Gamma,\Gamma_{new})}^{-1}(\cell)$.

\subsubsection{\bf An algorithm for generating the extensions}\label{ss:alg} First of all, with a fixed dual exponent $\cell=C^\Gamma_j+\sum_{\Gamma\setminus B^\Gamma_{j+1}} m_v E_v $ we associate the cycle $\ell':=C^{\Gamma_{new}}_{j+i(\Gamma,\Gamma_{new})}+\sum_{\Gamma\setminus B^{\Gamma}_{j+1}}m_v E_v \in L'(\Gamma_{new})$. Note that since all $\cell_{new}\in \calS'_{[Z^{\Gamma_{new}}_K]}$, it follows that $\ell'\leq s(\ell')\leq \cell_{new}$, where $s(\ell')$ is given by the generalized Laufer algorithm starting from $\ell'$, cf. \ref{ss:GLA}.  Moreover, the  integral cycles $s(\ell')-\ell'$ and $\cell_{new}-s(\ell')$ are both supported on the vertices $\{v_1,\dots,v_s\}$.  On the other hand, we also have the following identity
\begin{equation}\label{eq:inters1}
(\ell',E_w)_{\Gamma_{new}}=(\cell,E_w)_{\Gamma} \ \mbox{ for any }\ w\in \calv(\Gamma).
\end{equation}
Indeed, by the projection formula (\ref{eq:projform}) and Lemma \ref{lem:projC} one writes
$$(\ell',E_w)_{\Gamma_{new}}=(j^*_{\Gamma}(C^{\Gamma_{new}}_{j+i(\Gamma,\Gamma_{new})}), E_w)_{\Gamma} + \sum_{\Gamma\setminus B^{\Gamma}_{j+1}}m_v (E_v,E_w)_{\Gamma}=(\cell,E_w)_{\Gamma}.$$
Then, based on  the above discussion, we construct  computation sequences $\{x_t\}_{t\geq 0}$ which determine the cycles in  $\cpi_{(\Gamma,\Gamma_{new})}^{-1}(\cell)$ (provided that is not empty).
\vspace{0.2cm}

{\bf Step I.} \ \ We start with the above constructed cycle $x_0:=\ell'\in L'(\Gamma_{new})$ associated with $\cell$. If $\ell'\notin \calS'_{[Z_K^{\Gamma_{new}}]}$ then we proceed with the generalized Laufer algorithm in order to find $x_{t_0}=s(\ell')$. We emphasize that, by the assumption $\cpi_{(\Gamma,\Gamma_{new})}^{-1}(\cell)\neq \emptyset$, the property $|s(\ell')-\ell'|\subset \Gamma_{new}\setminus \Gamma$ must be satisfied. This implies that $(s(\ell'),E_v)_{\Gamma_{new}}=(\cell,E_v)_{\Gamma}$ for any $v\in \calv(\Gamma\setminus v_0)$ and $(\cell,E_{v_0})_{\Gamma}=(\ell',E_{v_0})_{\Gamma_{new}}\leq (s(\ell'),E_{v_0})_{\Gamma_{new}}\leq 0$. Therefore, $s(\ell')$ is either a dual exponent of $\Gamma_{new}$ or this property fails exactly at the vertex $v_0$. In any case, in the next step we will modify $s(\ell')$ in order to find all the possibilities.
\vspace{0.2cm}

{\bf Step II.} \ \ Note that $\delta_{v_0}^{\Gamma_{new}}=s+1$, hence a dual exponent $\cell_{new}$ satisfies $0\leq -(\cell_{new},E_{v_0})\leq s-1$. Therefore, for any integer $N$ in the interval
$$N\in [-(s(\ell'),E_{v_0})-s+1, -(s(\ell'),E_{v_0})]\cap \mathbb{Z}_{\geq 0},$$
and its partitions $\sum_{i=1}^s m_i=N$ for some $m_i\geq 0$ we associate the cycles
\begin{equation}\label{eq:dualextform}
 \cell_{new}(N;m_1,\dots, m_s):=s(\ell')+\sum_{i=1}^s m_i E_{v_i}
\end{equation}
 Then the claim, implied by the previous discussions,  is that the cycles $\cell_{new}(N;m_1,\dots, m_s)$ given by the different choices and partitions of $N$ provide the dual exponents in $\cpi_{(\Gamma,\Gamma_{new})}^{-1}(\cell)$. In particular, if the interval contains $N=0$, then this choice gives  $s(\ell')$  as a dual exponent in this case.

\subsubsection{\bf An example for non-extendable exponents}

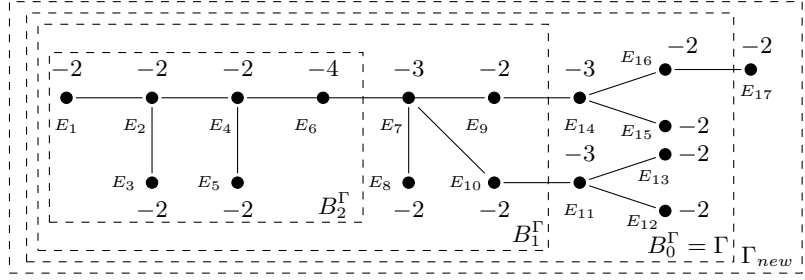
\begin{figure}[h!]
\begin{tikzpicture}[scale=.75]
\node (v1) at (-0.5,0) {};
\draw[fill] (-0.5,0) circle (0.1);
\node (v2) at (1,0) {};
\node (v3) at (1,-1.5) {};
\node (v4) at (2.5,0) {};
\node (v6) at (2.5,-1.5) {};
\node at (4,0) {};
\node (v7) at (5.5,0) {};
\node (v8) at (5.5,-1.5) {};
\node (v9) at (7,-1.5) {};
\node (v14) at (8.5,-1.5) {};
\node (v5) at (7,0) {};
\node (v10) at (8.5,0) {};
\node (v12) at (10,0.5) {};
\node (v17) at (11.5,0.5) {};
\node (v13) at (10,-0.5) {};
\node (v15) at (10,-1) {};
\node (v16) at (10,-2) {};

\draw[fill] (1,0) circle (0.1);
\draw[fill] (1,-1.5) circle (0.1);
\draw[fill] (2.5,0) circle (0.1);
\draw[fill] (2.5,-1.5) circle (0.1);
\draw[fill] (4,0) circle (0.1);
\draw[fill] (5.5,0) circle (0.1);
\draw[fill] (5.5,-1.5) circle (0.1);
\draw[fill] (7,-1.5) circle (0.1);
\draw[fill] (7,0) circle (0.1);
\draw[fill] (8.5,0) circle (0.1);
\draw[fill] (10,0.5) circle (0.1);
\draw[fill] (11.5,0.5) circle (0.1);
\draw[fill] (10,-0.5) circle (0.1);
\draw[fill] (8.5,-1.5) circle (0.1);
\draw[fill] (10,-1) circle (0.1);
\draw[fill] (10,-2) circle (0.1);

\draw  (v12) edge (v17);
\draw  (v15) edge (v14);
\draw  (v16) edge (v14);
\draw  (v9) edge (v14);
\draw  (v1) edge (v2);
\draw  (v2) edge (v3);
\draw  (v2) edge (v4);
\draw  (v4) edge (v5);
\draw  (v4) edge (v6);
\draw  (v7) edge (v8);
\draw  (v7) edge (v9);
\draw  (v5) edge (v10);
\draw  (v10) edge (v12);
\draw  (v10) edge (v13);
\node at (-0.5,0.5) {\small $-2$};
\node at (-0.5,-0.5) {\tiny $E_1$};
\node at (1,0.5) {\small $-2$};
\node at (0.7,-0.5) {\tiny $E_2$};
\node at (1,-2) {\small $-2$};
\node at (0.5,-1.5) {\tiny $E_3$};
\node at (2.5,0.5) {\small $-2$};
\node at (2.2,-0.5) {\tiny $E_4$};
\node at (2.5,-2) {\small $-2$};
\node at (2,-1.5) {\tiny $E_5$};
\node at (4,0.5) {\small $-4$};
\node at (3.7,-0.5) {\tiny $E_6$};
\node at (4.2,-1.9) {\small $B^\Gamma_2$};

\node at (5.5,0.5) {\small $-3$};
\node at (5.2,-0.5) {\tiny $E_7$};
\node at (5.5,-2) {\small $-2$};
\node at (5,-1.5) {\tiny $E_8$};
\node at (7,-2) {\small $-2$};
\node at (6.5,-1.5) {\tiny $E_{10}$};
\node at (7,0.5) {\small $-2$};
\node at (6.7,-0.5) {\tiny $E_9$};
\node at (7.6,-2.4) {\small $B_1^\Gamma$};

\node at (8.5,0.5) {\small $-3$};
\node at (8.5,-0.5) {\tiny $E_{14}$};
\node at (8.5,-1) {\small $-3$};
\node at (8.5,-2) {\tiny $E_{11}$};
\node at (10.5,-1) {\small $-2$};
\node at (10.5,-2) {\small $-2$};
\node at (9.6,-2.2) {\tiny $E_{12}$};
\node at (9.8,-1.4) {\tiny $E_{13}$};
\node at (10.3,0.9) {\small $-2$};
\node at (10.5,-0.5) {\small $-2$};
\node at (9.5,-0.6) {\tiny $E_{15}$};
\node at (9.5, 0.7) {\tiny $E_{16}$};
\node at (11.6,0.9) {\small $-2$};
\node at (11.6,0.1) {\tiny $E_{17}$};

\draw[dashed]  (-1.5,1.7) rectangle (12.4,-3.1);
\draw[dashed]  (-0.8,0.8) rectangle (4.7,-2.2);
\draw[dashed]  (-1,1.3) rectangle (7.95,-2.7);
\draw[dashed]  (-1.2,1.5) rectangle (11.2,-2.9);
\node at (10.4,-2.6) {\small $B_{0}^\Gamma=\Gamma$};
\node at (11.8,-2.8) {\small $\Gamma_{new}$};
\end{tikzpicture}
\caption{A small extension $\Gamma\subset \Gamma_{new}$ with non-extendable exponents} \label{fig:2}
\end{figure}

Let $\Gamma\subset \Gamma_{new}$ be the small extension shown in Figure \ref{fig:2} and consider the dual exponent $\cell=C_1^\Gamma+E_9+2E_{14}+E_{16}+4E_{15}$ from $Supp(P_0^\Gamma)$. One can check that $\cell=E^*_7+E^*_{11}+E^*_{14}$, hence $(\cell,E_{16})_{\Gamma}=0$, a property which will be crucial regarding its non-extendability.  Then we run the algorithm from section \ref{ss:alg} in order to find the possible extensions of $\cell$ to $\Gamma_{new}$.

In this case one has $i(\Gamma,\Gamma_{new})=0$. Hence, the starting point of the algorithm will be the cycle $\ell'=C^{\Gamma_{new}}_1+E_9+2E_{14}+E_{16}+4E_{15}$ associated with $\cell$. Since $C^{\Gamma_{new}}_1=s_{[Z_K]}^{\Gamma_{new}}+C_1^\Gamma$ and $s_{[Z_K]}^{\Gamma_{new}}=E^*_{17}$, a calculation shows that  $\ell'=\sum_{j=1}^{17}\ell'_j E_j$ is as follows:
\begin{center}
\begin{tikzpicture}[scale=.6]
\node (v1) at (-0.5,0) {\tiny $\frac{65}{31}$};
\node (v2) at (1,0) {\tiny $\frac{130}{31}$};
\node (v3) at (1,-1.5) {\tiny $\frac{65}{31}$};
\node (v4) at (2.5,0) {\tiny $\frac{130}{31}$};
\node (v6) at (2.5,-1.5) {\tiny $\frac{65}{31}$};
\node (v11) at (4,0) {\tiny $\frac{65}{31}$};
\node (v7) at (5.5,0) {\tiny $\frac{130}{31}$};
\node (v8) at (5.5,-1.5) {\tiny $\frac{65}{31}$};
\node (v9) at (7,-1.5) {\tiny $\frac{97}{31}$};
\node (v14) at (8.5,-1.5) {\tiny $\frac{64}{31}$};
\node (v5) at (7,0) {\tiny $\frac{132}{31}$};
\node (v10) at (8.5,0) {\tiny $\frac{134}{31}$};
\node (v12) at (10,0.5) {\tiny $\frac{79}{31}$};
\node (v17) at (11.5,0.5) {\tiny $\frac{24}{31}$};
\node (v13) at (10,-0.5) {\tiny $\frac{160}{31}$};
\node (v15) at (10,-1.2) {\tiny $\frac{32}{31}$};
\node (v16) at (10,-2) {\tiny $\frac{32}{31}$};

\draw  (v12) edge (v17);
\draw  (v15) edge (v14);
\draw  (v16) edge (v14);
\draw  (v9) edge (v14);
\draw  (v1) edge (v2);
\draw  (v2) edge (v3);
\draw  (v2) edge (v4);
\draw  (v4) edge (v11);
\draw  (v7) edge (v11);
\draw  (v7) edge (v5);
\draw  (v4) edge (v6);
\draw  (v7) edge (v8);
\draw  (v7) edge (v9);
\draw  (v5) edge (v10);
\draw  (v10) edge (v12);
\draw  (v10) edge (v13);
\end{tikzpicture}
\end{center}
where the coefficients $\ell'_j$ are shown in the place of the corresponding vertices. Using this explicit calculation, one can see the validity of the identity $(\ell',E_v)_{\Gamma_{new}}=(\cell,E_v)_{\Gamma}$ for any $v\in\calv(\Gamma)$, cf. (\ref{eq:inters1}). Moreover, for the new vertex one gets $(\ell',E_{17})_{\Gamma_{new}}=1$, hence $\ell'\notin \calS'(\Gamma_{new})$. Therefore, the Laufer algorithm targeting the cycle $s(\ell')$ contructs at the first step the cycle $x_1=\ell'+E_{17}$ for which we have $(x_1,E_{16})_{\Gamma_{new}}=1>0$. Hence, the algorithm must continue with the cycle $x_1+E_{16}$. Therefore we get $|s(\ell')-\ell'|\nsubseteq \Gamma_{new}\setminus \Gamma$ which implies that there are no extensions for $\cell$.
We emphasize that the extendibility was caused by the property $(\cell,E_{16})_{\Gamma}=(\ell',E_{16})_{\Gamma_{new}}=0$.

\subsection{Counting coefficients of the extensions}

\subsubsection{}
In order to simplify the notations, for a dual exponent $\cell$ associated with a graph $\Gamma$  let  $z(\cell)$ be  the coefficient of $\cell$ in the corresponding Poincar\'e series $Z_{[Z_K^{\Gamma}]}$.

We consider the small extension of elliptic graphs $\Gamma\subset \Gamma_{new}$ discussed in the previous section, with the assumption that  $\Gamma_{new}$ is contructed from $\Gamma$ by attaching new vertices $\{v_1,\dots,v_s\}$ to an end-vertex $v_0$ of $\Gamma$.

Let $\cell \in Supp_j(\check{P}_0^{\Gamma})$ be a dual exponent of $\Gamma$ which can be extended to $\Gamma_{new}$ and   consider all its extensions $\cell_{new}\in \cpi^{-1}_{(\Gamma,\Gamma_{new})}(\cell)\subset  Supp_{j+i(\Gamma,\Gamma_{new})}(\check{P}_0^{\Gamma_{new}})$.  Then, the aim of this section is to prove the following identity for the coefficients.

\begin{thm}\label{thm:ccee}
\begin{equation}\label{eq:zsum1}
\sum_{\cell_{new}\in \cpi^{-1}_{(\Gamma,\Gamma_{new})}(\cell)} z(\cell_{new})=z(\cell).
\end{equation}
\end{thm}

\begin{proof}

We may write the fixed dual exponent in the following form:  $\cell= \sum_{v\in \mathcal{N}\cup \mathcal{E}} \cell^*_v E_v^*$, where $\cell^*_v=-(\cell, E_v)_\Gamma$, $\caln$ denotes the set of nodes and $\cale$ is the set of end-vertices of $\Gamma$.  Then, by (\ref{eq:z}) its coefficient in  $Z_{[Z_K^{\Gamma}]}$ equals $z(\cell)=(-1)^{\sum_{v\in \mathcal{N}}\cell_v^*}\prod_{v\in \mathcal{N}}\binom{\delta_v-2}{\cell^*_v}$. Note that $v_0$ does not contribute to $z(\cell)$, since by assumption $v_0$ is an end-vertex of $\Gamma$. On the other hand, by Lemma \ref{lem:cpi} one has $(\cell, E_v)_\Gamma=(\cell_{new},E_v)_{\Gamma_{new}}$ for any $v\in \mathcal{V}(\Gamma)\setminus v_0$ and $\cell_{new}\in \cpi^{-1}_{(\Gamma,\Gamma_{new})}(\cell)$. Moreover, the vertices $v_i$ $(i\in\{1,\dots,s\})$ do not contribute to $z(\cell_{new})$ since they are end-vertices of $\Gamma_{new}$. Therefore, $z(\cell_{new})=z(\cell)\cdot z_{v_0}(\cell_{new})$, where
\begin{equation}\label{eq:zv0}
z_{v_0}(\cell_{new})=(-1)^{-(\cell_{new}, E_{v_0})_{\Gamma_{new}}}\cdot\binom{s-1}{-(\cell_{new}, E_{v_0})_{\Gamma_{new}}}
\end{equation}
is the contribution of $v_0$ to $z(\cell_{new})$. This reduces (\ref{eq:zsum1}) to the following identity:
\begin{equation}\label{eq:sumzv0}
 \sum_{\cell_{new}\in \cpi^{-1}_{(\Gamma,\Gamma_{new})}(\cell)} z_{v_0}(\cell_{new})=1.
\end{equation}

Now, consider the cycles $\ell', s(\ell')\in L'(\Gamma_{new})$ associated with $\cell$ by the algorithm developed in section \ref{ss:extcompseq}. We have also proved that $(s(\ell'),E_v)_{\Gamma_{new}}=(\cell,E_v)_{\Gamma}$ for every $v\in \mathcal{V}(\Gamma)\setminus v_0$. In the case of $v_0$, we set the notation $m:=-(s(\ell'),E_{v_0})_{\Gamma_{new}}$ for simplicity. Then, by (\ref{eq:dualextform}) the extensions are given by $\cell_{new}(N,m_i)=s(\ell')+\sum_{i=1}^s m_i E_{v_i}$ for all possible $m_i\geq 0$ such that $\sum_{i=1}^s m_i=N$ and $N\in [m-s+1,m]\cap \Z_{\geq 0}$. Hence, for fixed $N$ and $m_i$, (\ref{eq:zv0}) reads as
$$z_{v_0}(\cell_{new}(N,m_i))=(-1)^{m-N}\cdot\binom{s-1}{m-N}.$$

On the other hand, the number of partitions $\sum_{i=1}^s m_i=N$ for a fixed $N$ is counted by the Ehrhart polynomial $\mathcal{L}_{\Delta_{s-1}}(N)$ of the standard $(s-1)$-simplex $\Delta_{s-1}$, whose explicit expression is given by $\mathcal{L}_{\Delta_{s-1}}(N)=\binom{N+s-1}{s-1}$, see \cite[Theorem 2.2]{BR} for more details.
This implies that  (\ref{eq:sumzv0}) is equivalent with the identity
$$\sum_{N\in [m-s+1,m]\cap \Z_{\geq 0}} (-1)^{m-N}\cdot\binom{s-1}{m-N}\cdot \binom{N+s-1}{s-1}=1, $$
which can be deduced by part (a) of the following lemma.
\end{proof}

\begin{lemma}\label{lem:binomial} (a) For any non-negative integers $d,m\in \Z_{\geq 0}$ one has
$$S_{d,m}=\sum_{i=0}^{\min\{d,m\}}  (-1)^{i}\cdot\binom{d}{i}\cdot \binom{m+d-i}{d}=1.$$
(b) Assume $m\leq d$. Then
$$S'_{d,m}=\sum_{i=0}^m (-1)^i\cdot \binom{d}{i}\cdot \binom{m+d-i-1}{d-1}=0.$$
\end{lemma}

\begin{proof}
By using the binomial expression $\binom{p}{r}=\binom{p-1}{r}+\binom{p-1}{r-1}, p>r$ consequently, it can be proved that $S_{d,m}=S_{d,m-1}+S_{d-1,m}-S_{d-1,m-1}$. Then, in case (a) one can proceed easily by induction. Using again the same binomial expression, (b) follows from (a).
\end{proof}

\subsection{Counting coefficients of non-extendable exponents}

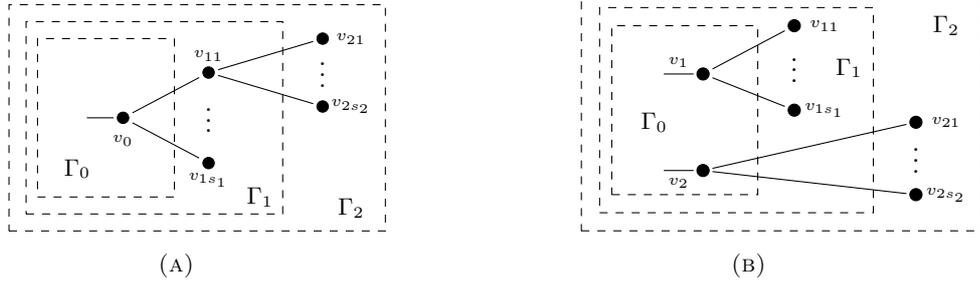
\begin{figure}[h!]
\centering
\begin{subfigure}[b]{0.3\textwidth}
\centering
\begin{tikzpicture}[scale=.75]
\node (v0) at (1,-0.6) {};
\draw[fill] (v0) circle (0.1);
\draw[dashed]  (-0.5,0.8) rectangle (1.9,-2);
\node at (0.2,-1.5) {\small $\Gamma_0$};
\node at (1,-1) {\tiny $v_{0}$};
\node (v11) at (2.5,0.2) {};
\node (v12) at (2.5,-1.4) {};
\draw[fill] (v11) circle (0.1);
\draw[fill] (v12) circle (0.1);
\draw  (v0) edge (v11);
\draw  (v0) edge (v12);
\node at (2.5,-0.5) {$\vdots$};
\node at (2.5,0.5) {\tiny $v_{11}$};
\node at (2.5,-1.7) {\tiny $v_{1s_1}$};
\draw[dashed]  (-0.7,1.1) rectangle (3.8,-2.3);
\node at (3.4,-2) {\small $\Gamma_1$};
\node (v21) at (4.5,0.8) {};
\node (v22) at (4.5,-0.4) {};
\draw[fill] (v21) circle (0.1);
\draw[fill] (v22) circle (0.1);
\draw  (v11) edge (v21);
\draw  (v11) edge (v22);
\node at (4.5,0.3) {$\vdots$};
\node at (5,0.8) {\tiny $v_{21}$};
\node at (5,-0.4) {\tiny $v_{2s_2}$};
\draw[dashed]  (-1,1.4) rectangle (5.6,-2.6);
\node at (5,-2.2) {\small $\Gamma_2$};
\node (v00) at (0.2,-0.6) {};
\draw  (v00) edge (v0);
\end{tikzpicture}
  \caption{}
  \label{fig:sub1}
\end{subfigure}%
\hspace{3cm}
\begin{subfigure}[b]{0.3\textwidth}
  \centering
\begin{tikzpicture}[scale=.8]
\node (v0) at (1,0) {};
\draw[fill] (v0) circle (0.1);
\draw[dashed]  (-0.5,0.8) rectangle (1.9,-2);
\node at (0.2,-0.8) {\small $\Gamma_0$};
\node at (0.6,0.2) {\tiny $v_{1}$};
\node (v11) at (2.5,0.8) {};
\node (v12) at (2.5,-0.6) {};
\draw[fill] (v11) circle (0.1);
\draw[fill] (v12) circle (0.1);
\draw  (v0) edge (v11);
\draw  (v0) edge (v12);
\node at (2.5,0.2) {$\vdots$};
\node at (3,0.8) {\tiny $v_{11}$};
\node at (3,-0.6) {\tiny $v_{1s_1}$};
\draw[dashed]  (-0.7,1.1) rectangle (3.8,-2.3);
\node at (3.4,0.1) {\small $\Gamma_1$};
\node (v00) at (0.2,0) {};
\draw  (v00) edge (v0);

\node (v1) at (1,-1.6) {};
\draw[fill] (v1) circle (0.1);
\node at (0.6,-1.8) {\tiny $v_{2}$};
\node (v21) at (4.5,-0.8) {};
\node (v22) at (4.5,-2) {};
\draw[fill] (v21) circle (0.1);
\draw[fill] (v22) circle (0.1);
\draw  (v1) edge (v21);
\draw  (v1) edge (v22);
\node at (4.5,-1.3) {$\vdots$};
\node at (5,-0.8) {\tiny $v_{21}$};
\node at (5,-2) {\tiny $v_{2s_2}$};
\draw[dashed]  (-1,1.4) rectangle (5.6,-2.6);
\node at (5,0.8) {\small $\Gamma_2$};
\node (v01) at (0.2,-1.6) {};
\draw  (v01) edge (v1);
\end{tikzpicture}
  \caption{}
  \label{fig:sub2}
\end{subfigure}
\caption{The two cases for $\Gamma_0\subset\Gamma_1\subset\Gamma_2$.}
\label{fig:test}
\end{figure}

Let $\Gamma_0\subset \Gamma_1\subset \Gamma_2$ be elliptic graphs such that for all $i\in\{0,1\}$ the extensions $\Gamma_i\subset \Gamma_{i+1}$ are small and the vertices of $\Gamma_{i+1}\setminus \Gamma_i$ are attached to an end-vertex of $\Gamma_i$.

We fix a dual exponent $\cell$ of $\Gamma_0$. \emph{Assume that $\cell$ can be extended to $\Gamma_2$, ie. $\cpi_{(\Gamma_0,\Gamma_2)}^{-1}(\cell)\neq \emptyset$.} Since by definition one gets $\cpi_{(\Gamma_0,\Gamma_2)}=\cpi_{(\Gamma_0,\Gamma_1)}\circ \cpi_{(\Gamma_1,\Gamma_2)}$, it implies that $\cell$ is extendable to $\Gamma_1$ as well, and there exists at least one extension of $\cell$ on $\Gamma_1$ which must be extendable to $\Gamma_2$. On the other hand, there might exist extensions  of $\cell$ on $\Gamma_1$ which are non-extendable to $\Gamma_2$. Their set will be denoted by
$$\cpi^{-1}_{(\Gamma_0,\Gamma_1)}(\cell)_{ne}\subset \cpi^{-1}_{(\Gamma_0,\Gamma_1)}(\cell).$$

The aim of this subsection is to prove the following identity regarding the sum of the coefficients of these non-extendable exponents.

\begin{theorem}\label{thm:nedualexp}
\begin{equation}
 \sum_{\cell_1\in \cpi^{-1}_{(\Gamma_0,\Gamma_1)}(\cell)_{ne}} z(\cell_1)=0.
\end{equation}
\end{theorem}

\begin{proof}
 If $\cpi^{-1}_{(\Gamma_0,\Gamma_1)}(\cell)_{ne}=\emptyset$ then we set the above sum to be zero and the theorem is automatically true.  Therefore we can assume that there exists an extension $\cell_1\in \cpi^{-1}_{(\Gamma_0,\Gamma_1)}(\cell)_{ne} \subset Supp(\check{P}_0^{\Gamma_1})$ of $\cell$ such that $\cpi^{-1}_{(\Gamma_1,\Gamma_2)}(\cell_1)=\emptyset$.  If we associate with $\cell_1$ the cycles $\ell', s(\ell')\in L'(\Gamma_2)$ according to the algorithm \ref{ss:alg}, then our assumption implies that $|s(\ell')-\ell'|\nsubseteq \Gamma_2\setminus \Gamma_1$. Indeed, otherwise step II. of \ref{ss:alg} would find an element in $\cpi^{-1}_{(\Gamma_1,\Gamma_2)}(\cell_1)$.

 Next, we analyze the generalized Laufer computation sequence targeting $s(\ell')$ in more details.

Note that $\Gamma_2$ is contructed from $\Gamma_1$ by attaching new vertices $v_{21},\dots, v_{2s_2}$ to an end-vertex of $\Gamma_1$. Since $(\ell',E_w)_{\Gamma_2}=(\cell_1,E_w)_{\Gamma_1}\leq 0$ for every $w\in \calv(\Gamma_1)$, we must have $(\ell',E_{v_{2j}})_{\Gamma_2}>0$ for some $j\in\{1,\dots,s_2\}$, otherwise $\ell'=s(\ell')$ would be an extension of $\cell_1$. Hence, we can choose the generalized Laufer computation sequence ($\{x_t\}_{t\geq 0}, x_0:=\ell'$) in such a way that after some steps we reach a cycle  $x_{t_1}=\ell'+\sum_{j=1}^{s_2} m_j E_{v_{2j}}$ for some $m_j\geq 0$, where $(x_{t_1},E_{v_{2j}})_{\Gamma_2}\leq 0$ for any $j$. Then, by our assumption, the algorithm turns back to the end-vertex of $\Gamma_1$ to which the vertices $v_{2j}$ are connected. As illustrated by Figure \ref{fig:test}, we have two cases depending whether this end-vertex belongs to $\Gamma_0$ or not.

First we look at case (B). By the previous argument we have $(x_{t_1},E_{v_2})_{\Gamma_2}>0$ and the algorithm proceeds with $x_{t_1+1}=x_{t_1}+E_{v_2}$, which means that $\cell_1$ does not extend to $\Gamma_2$. On the other hand, if we consider any other extension $\cell_1'$ of $\cell$, one knows by Lemma \ref{lem:cpi} that in this case $(\cell_1',E_{v_2})_{\Gamma_1}=(\cell_1,E_{v_2})_{\Gamma_1}=(\cell,E_{v_2})_{\Gamma_0}$. Therefore,  the generalized Laufer computation sequence used in the algorithm \ref{ss:alg} for the extension of $\cell_1'$ to $\Gamma_2$ can be chosen in the same way and has the same properties as in the case of $\cell_1$, which implies that neither $\cell_1'$ can be extended to $\Gamma_2$. This contradicts to the assumption, hence we must have $\cpi^{-1}_{(\Gamma_0,\Gamma_1)}(\cell)_{ne}=\emptyset$ for case (B).

In case (A), first we set the notation $n:=-(\cell_1,E_{v_{11}})_{\Gamma_1}=-(\ell',E_{v_{11}})_{\Gamma_2}$. The assumption that $\cell=\cpi_{(\Gamma_0,\Gamma_1)}(\cell_1)$ can be extended to $\Gamma_2$ implies that $N_1:=-(\cell_1,E_{v_0})_{\Gamma_1}>0$ and there exists another extension of $\cell$  greater or equal to $\cell_1+tE_{v_{11}}$ for some $t>0$, which will be extendable to $\Gamma_2$. Moreover, since $\cell_1$ is a dual exponent and the valency of $v_{11}$ is $s_1+1$, we have that $N_1\in (0,s_1-1]$. By changing $\cell_1$ if needed, we may assume that $N_1$ of $\cell_1$ is the maximal amongst all those extensions $\cell_1'$ which are non-extendable to $\Gamma_2$ and $-(\cell_1',E_{v_{11}})_{\Gamma_1}=n$ is fixed.
Then, associated with $n$ all these dual exponents can be constructed as
$$\cell_1(N;n_2,\dots,n_{s_1})=\cell_1+\sum_{j=2}^{s_1}n_j E_{v_{1j}}$$ such that $\sum_{j=2}^{s_1}n_j=N$ for $N\leq N_1$.  In particular, one has $\cell_1(0;0,\dots,0)=\cell_1$.

By \ref{lem:cpi} we have $(\cell_1(N;n_2,\dots,n_{s_1}),E_v)_{\Gamma_1}=(\cell,E_v)_{\Gamma_0}$ for every $v\in \calv(\Gamma_0)\setminus  v_{11}$.  Therefore, the coefficients in the series  $Z_{[Z_K^{\Gamma_1}]}$ of the dual exponents $\cell_1(N;n_2,\dots,n_{s_1})$ are  expressed as
\begin{equation*}
z(\cell_1(N;n_2,\dots,n_{s_1}))=z(\cell)\cdot (-1)^{N_1-N}\binom{s_1-1}{N_1-N}.
\end{equation*}
Hence, if we count all these coefficients we can write
\begin{equation}\label{eq:ccnem}
S(n):=\sum_{0\leq N\leq N_1 \atop \sum_{j=2}^{s_1}n_j=N} z(\cell_1(N;n_2,\dots,n_{s_1}))=z(\cell)\cdot \sum _{0\leq N\leq N_1 \atop \sum_{j=2}^{s_1}n_j=N} (-1)^{N_1-N}\binom{s_1-1}{N_1-N}.
\end{equation}
Similarly as in the proof of Theorem \ref{thm:ccee}, the number of partitions $(n_2,\dots,n_{s_1})$ of  $N$ is expressed by the Ehrhart polynomial $\mathcal{L}_{\Delta_{s_{1}-2}}(N)$ of the standard $(s_1-2)$-simplex $\Delta_{s_{1}-2}$, which is $\mathcal{L}_{\Delta_{s_{1}-2}}(N)=\binom{N+s_1-2}{s_1-2}$, cf. \cite{BR}. This implies that
$$S(n)=z(\cell)\cdot \sum_{N=0}^{N_1}(-1)^{N_1-N}\binom{s_1-1}{N_1-N}\binom{N+s_1-2}{s_1-2}.$$
Since the summation in the above expression vanishes by Lemma \ref{lem:binomial}, we get $S(n)=0$.

Finally, followed by the above discussion one writes
$$\sum_{\cell_1\in \cpi^{-1}_{(\Gamma_0,\Gamma_1)}(\cell)_{ne}} z(\cell_1)=\sum_{\mbox{\tiny all possible} \ n} S(n)=0$$
which proves the result.

\end{proof}

\subsection{The general case} \label{ss:gencase}
Following the concepts and notations of section \ref{s:6}, we consider two elliptic graphs $\Gamma'\subset \Gamma$. As before we assume  that
\begin{equation}\label{eq:condext}
\mbox{every} \  v\in \partial_{\Gamma}(\Gamma') \ \mbox{is an end-vertex of} \ \Gamma'.
\end{equation}
Recall that in this case  Theorem \ref{thm:dualmapinj} applies and for any exponent $\ell\in Supp_j(P_0^\Gamma)$ such that $j\geq i_{(\Gamma',\Gamma)}-1$, where $i_{(\Gamma',\Gamma)}$ is the index of the extension (cf. (\ref{eq:ext})), we get $\pi_{(\Gamma',\Gamma)}(\ell)\in Supp_{j-i_{(\Gamma',\Gamma)}}(P_0^{\Gamma'})$.

Now, let $\ell'\in Supp_{j}(P_0^{\Gamma'})$ be an exponent of $\Gamma'$ and \emph{assume that $\ell'$ can be extended to $\Gamma$}. That is  $\emptyset\neq \pi_{(\Gamma',\Gamma)}^{-1}(\ell')\subset  Supp_{j+i_{(\Gamma',\Gamma)}}(P_0^{\Gamma})$. Then, using the previous  sections we can prove the following identity.

\begin{theorem}\label{thm:cc} For an extension $\Gamma'\subset \Gamma$ as above and an exponent $\ell'\in Supp_{j}(P_0^{\Gamma'})$ which is extendable to $\Gamma$ one has
 \begin{equation*}
\sum_{\ell \in \pi_{(\Gamma',\Gamma)}^{-1}(\ell')\neq \emptyset}w(\ell)=w(\ell'),
 \end{equation*}
where $w(\cdot)$ denotes the coefficient of an exponent in the canonical polynomial associated with the corresponding graph.
\end{theorem}

\begin{remark}
 Note that in `dual' terms, the identity from Theorem \ref{thm:cc} is equivalent with $$\sum_{\cell\in \cpi^{-1}_{(\Gamma',\Gamma)}(\cell')\neq \emptyset} z(\cell)=z(\cell'),$$ where $\cell'$ (resp. $\cell$) is the dual exponent of $\ell'$ (resp. $\ell$).
\end{remark}

\begin{proof}[Proof of Theorem \ref{thm:cc}]
 First of all, we note that one can construct a sequence of small extensions such that $\Gamma'=\Gamma_0\subset\Gamma_1\subset\dots\subset\Gamma_n=\Gamma$ for some $n\geq 0$ and $\Gamma_{i+1}\setminus \Gamma_{i}$ is attached to an end-vertex of $\Gamma_i$, for all $i$.

 Accordingly, the projection map $\pi_{(\Gamma',\Gamma)}$ can be decomposed as $\pi_{(\Gamma',\Gamma)}=\pi_{(\Gamma_0,\Gamma_1)}\circ\dots\circ \pi_{(\Gamma_{n-1},\Gamma_n)}$, hence for the dualized map one also writes  $\cpi_{(\Gamma',\Gamma)}=\cpi_{(\Gamma_0,\Gamma_1)}\circ\dots\circ \cpi_{(\Gamma_{n-1},\Gamma_n)}$.

 Continuing to use the dual terms in the present proof, let $\cell'\in Supp(\check{P}_0^{\Gamma'})$ be a dual exponent of $\Gamma'$ which is extendable to $\Gamma$. For any $i\in\{0,\dots,n-1\}$ the set of extensions of $\cell'$ to $\Gamma_i$ decomposes into
 \begin{equation*}
  \emptyset\neq\cpi^{-1}_{(\Gamma',\Gamma_i)}(\cell')=\cpi^{-1}_{(\Gamma',\Gamma_i)}(\cell')_e \sqcup \cpi^{-1}_{(\Gamma',\Gamma_i)}(\cell')_{ne},
 \end{equation*}
where the elements of $\cpi^{-1}_{(\Gamma',\Gamma_i)}(\cell')_e\neq \emptyset$ are extendable to $\Gamma_{i+1}$, while the  dual exponents in $\cpi^{-1}_{(\Gamma',\Gamma_i)}(\cell')_{ne}$ can not be further extended.

By Theorem \ref{thm:nedualexp} it follows that
\begin{equation}\label{eq:gennedualexp}
\sum_{\cell\in \cpi^{-1}_{(\Gamma',\Gamma_{i})}(\cell')_{ne}} z(\cell)=0
\end{equation}
since one expresses $\cpi^{-1}_{(\Gamma',\Gamma_{i})}(\cell')_{ne}=\cup_{\cell_{\Gamma_{i-1}}\in \cpi^{-1}_{(\Gamma',\Gamma_{i-1})}(\cell')}\cpi^{-1}_{(\Gamma_{i-1},\Gamma_{i})}(\cell_{\Gamma_{i-1}})_{ne}$.
Moreover,  this implies the following identities
\begin{align}
\sum_{\cell_{\Gamma_{i+1}}\in \cpi^{-1}_{(\Gamma',\Gamma_{i+1})}(\cell')} z(\cell_{\Gamma_{i+1}})
= \sum_{\cell_{\Gamma_{i}}\in \cpi^{-1}_{(\Gamma',\Gamma_{i})}(\cell')_e} \
\sum_{\cell\in \cpi^{-1}_{(\Gamma_i,\Gamma_{i+1})}(\cell_{\Gamma_i})} z(\cell)\\
\nonumber \stackrel{\textrm{Thm. \ref{thm:ccee}}}{=} \sum_{\cell_{\Gamma_{i}}\in \cpi^{-1}_{(\Gamma',\Gamma_{i})}(\cell')_e} z(\cell_{\Gamma_{i}})\stackrel{(\ref{eq:gennedualexp})}{=}
\sum_{\cell_{\Gamma_{i}}\in \cpi^{-1}_{(\Gamma',\Gamma_{i})}(\cell')} z(\cell_{\Gamma_{i}}).
\end{align}
Hence,  an induction argument finishes the proof.
\end{proof}

\begin{corollary}
Let $\Gamma$ be an elliptic graph and consider its NN-elliptic sequence $\{B_j,C_j\}_{j=-1}^m$. In particular, we look at the extension $B_{j+1}\subset \Gamma$. Since $B_{j+1}$ is numerically Gorenstein, every $v\in \partial_{\Gamma}(B_{j+1})$ is an end-vertex of $B_{j+1}$. Moreover, the only exponent in $Supp_{-1}(P_0^{B_{j+1}})$ is $Z_K^{B_{j+1}}-E^{B_{j+1}}$ with coefficient $1$. This extends to $\Gamma$, and its extensions are all the exponents from $Supp_j(P_0^{\Gamma})$. Then one deduces the followings:
\begin{itemize}
\item[(1)] Theorem \ref{thm:cc} implies $\sum_{\ell \in Supp_j(P_0^{\Gamma})} w(\ell)=1$.  Note that this way we get back the statement of Lemma \ref{lem:coefffunc}(b), and one can think of Theorem \ref{thm:cc} as a refinement of it.
\item[(2)] By considering a sequence of small extensions from $B_{j+1}$ to $\Gamma$ and using inductively the algorithm from section \ref{ss:alg} one can generate all the exponents of $\Gamma$ in
$Supp_j(P_0^{\Gamma})$ for $j\geq 0$.
\end{itemize}
\end{corollary}

\section{Good extensions and an inclusion type formula for canonical polynomials}\label{s:7}

In this section we will introduce the notion of good extensions and characterize them with a natural inclusion type formula for the corresponding canonical polynomials.

\subsection{} Let $\Gamma'\subset \Gamma$ be an extension of elliptic graphs such that (\ref{eq:condext}) holds, ie. $\Gamma\setminus \Gamma'$ is attached to $\Gamma'$ through end-vertices of $\Gamma'$, cf. sect. \ref{ss:gencase}.

\begin{definition}
We say that $\Gamma'\subset \Gamma$ is a \emph{good extension} if every $\ell'\in Supp(P_0^{\Gamma'})$ can be extended to an exponent of $\Gamma$.
\end{definition}

\begin{example}
 Let $\Gamma$ be an elliptic graph and  $\{B_j;C_j\}_{j=-1}^m$ its NN-elliptic sequence. Then the extension $B_m\subset \Gamma$ is always good. Indeed, the only exponent of the minimal elliptic graph $B_m$ is $Z_K^{B_m}-E^{B_m}=C-E^{B_m}$ which can be realized as the projection of any of the exponents from $Supp_{m-1}(P_0^\Gamma)$.
\end{example}

 Now, let $\Gamma$ be an elliptic graph, $\{B_j;C_j\}_{j=-1}^m$ its associated NN-elliptic sequence  and let $P_0^{\Gamma}(\bt)=\sum_{\ell}w(\ell)\bt^{\ell}$ be its canonical polynomial. Then for any $i\in\{0,\dots,m\}$ we define the following truncated polynomials:
\begin{align}\label{ipol}
 \calF_iP_{0}^{\Gamma}(\bt):=\sum_{\ell\in Supp(P_{0}^{\Gamma})\atop \ell|_{\Gamma\setminus B_i}<0} w(\ell) \bt^\ell.
\end{align}
(One can also set $\calF_{-1}P_{0}^{\Gamma}(\bt):=P_{0}^{\Gamma}(\bt)$ for convenience.)

By Corollary \ref{cor:sign}, this means that for a fixed $i\geq 0$, $\calF_iP_{0}^{\Gamma}$ consists of the monomials with exponents $\ell\in Supp_j(P_{0}^{\Gamma})$ where $i-1\leq j\leq m-1$. In particular, this implies that
\begin{equation}\label{eq:partic1}
\calF_{0}P_{0}^{\Gamma}(\bt)=P_{0}^{\Gamma}.
\end{equation}

Then we prove the following statement.


\begin{theorem}\label{thm:mainrec} The extension $\Gamma'\subset \Gamma$ is good if and only if the corresponding canonical polynomials satisfy the following identity:
 $$\calF_{i_{(\Gamma',\Gamma)}}P_{0}^{\Gamma}(\bt_{\Gamma'})=P_{0}^{\Gamma'}(\bt),$$
 where $i_{(\Gamma',\Gamma)}$ denotes the index of the extension as in section \ref{ss:index}.
\end{theorem}

\begin{proof}
First of all, note that the identity implies that any exponent of $P_{0}^{\Gamma'}$ is realized in  $P_{0}^{\Gamma}(\bt_{\Gamma'})$ as the projection of at least one exponent of $\Gamma$, hence the extension is good.

Conversely, assume $\Gamma'\subset \Gamma$ is a good extension. Then by Theorem \ref{thm:dualmapinj} the map $$\pi_{(\Gamma',\Gamma)}|_{Supp_j(P_0^\Gamma)}:Supp_j(P_0^\Gamma)\to Supp_{j-i_{(\Gamma',\Gamma)}}(P_0^{\Gamma'})$$ is well defined and surjective for any $j\geq i_{(\Gamma',\Gamma)}-1$. Hence, the extensions of the exponents of $P_0^{\Gamma'}$ to $\Gamma$ are the exponents in the set $\cup_{j\geq i_{(\Gamma',\Gamma)}-1}Supp_j(P_0^{\Gamma})$. Thus, in order to compare the two canonical polynomial, one has to consider only those  monomials in $P_0^\Gamma$ which are exactly in the truncated part $\calF_{i_{(\Gamma',\Gamma)}}P_0^{\Gamma}$.

 On the other hand, for any $\ell'\in Supp(P_0^{\Gamma'})$ we have proved in Theorem \ref{thm:cc} that $$\sum_{\ell \in \pi_{(\Gamma',\Gamma)}^{-1}(\ell')}w(\ell)=w(\ell').$$ This implies that the sum of monomials $\sum_{\ell \in \pi_{(\Gamma',\Gamma)}^{-1}(\ell')}w(\ell)\bt^{\ell}_{\Gamma}$ in $\calF_{i_{(\Gamma',\Gamma)}}P_0^{\Gamma}$ reduced to the variables of the vertices of $\Gamma'$ gives exactly the monomial $w(\ell')\bt^{\ell'}_{\Gamma'}$ in $P_0^{\Gamma'}$.
\end{proof}

\begin{corollary}\label{cor:mainrec}
Let $\{B_j;C_j\}_{j=-1}^m$ be the NN-elliptic sequence of some $\Gamma$. If $B_j\subset \Gamma$ is a good extension, then
$$\calF_{j}P_{0}^{\Gamma}(\bt_{B_j})=P_{0}^{B_j}(\bt).$$
\end{corollary}

\begin{remark}
(1) \ In particular, if $B_0\subset \Gamma$ is good, then  Corollary \ref{cor:mainrec} and (\ref{eq:partic1}) imply that the canonical polynomial of $\Gamma$ reduced to the vertices of $B_0$ is the same as the  canonical polynomial of  $B_0$, ie. $P_{0}^{\Gamma}(\bt_{B_0})=P_{0}^{B_0}(\bt)$.\\
(2) \ Note that for $j=m$ one deduces the equality $\calF_mP_{0}^{\Gamma}(\bt_{B_m})=P_{0}^{B_m}(\bt)=\bt^{C-E}$, where $B_m=|C|$ is the support of the minimal elliptic cycle $C$. (See also Theorem \ref{thm:redminell}.)  \\
\end{remark}

\subsection{An example}
Let $\Gamma$ be the subgraph of the elliptic graph considered in Figure \ref{fig:2}, generated by the vertices $\{E_i\}_{i=1}^{12}$ and the corresponding edges. Then $\Gamma$ is non-numerically Gorenstein and its NN-elliptic sequence consists of the following subgraphs:
$$B_{-1}^{\Gamma}=\Gamma \supset B_0^{\Gamma}=\langle E_i\rangle_{i=1}^{10} \supset B_1^{\Gamma}=\langle E_i\rangle_{i=1}^{6},$$
where $\langle\cdot \rangle$ denotes the subgraph generated by the corresponding vertices.

These informations can also be extracted from the canonical polynomial of $\Gamma$ which is as follows:
\begin{align*}
P_0^{\Gamma}(\bt)=&\bt^{(1,3,1,3,1,1,1,0,0,0,-1,-3)} - 2\bt^{(0,1,0,1,0,0,-1,-1,-2,-1,-1,-2)}+\bt^{(0,1,0,1,0,0,-1,-1,-2,-2,-3,-7)}\\
&-2\bt^{(0,1,0,1,0,0,-1,-1,-1,-2,-3,-7)}+\bt^{(0,1,0,1,0,0,-1,-1,-1,-3,-5,-12)}\\
&+\bt^{(0,1,0,1,0,0,-1,-2,-2,-1,-1,-2)}+\bt^{(0,1,0,1,0,0,-1,-1,-1,-1,-1,-2)}\\
&+\bt^{(0,1,0,1,0,0,-1,-2,-1,-2,-3,-7)}+\bt^{(0,1,0,1,0,0,-1,-1,-3,-1,-1,-2)}\\
&+\bt^{(0,1,0,1,0,0,-1,-3,-1,-1,-1,-2)}-2\bt^{(0,1,0,1,0,0,-1,-2,-1,-1,-1,-2)}.
\end{align*}
In the case of the extension $B_0^{\Gamma}\subset \Gamma$ we have $i_{(B_0^{\Gamma},\Gamma)}=0$. Moreover, since $B_0^{\Gamma}$ is the graph for which we have calculated the canonical polynomial in Example \ref{ex:2}, one can see that
$P_{0}^{\Gamma}(\bt_{B_0^{\Gamma}})=P_{0}^{B_0^{\Gamma}}(\bt)$, which says that the extension $B_0^{\Gamma}\subset \Gamma$ is good.

\end{document}